\documentclass[11pt]{amsart}

\usepackage[colorlinks=true, pdfstartview=FitV, linkcolor=blue, 
citecolor=blue]{hyperref}

\usepackage{a4wide}
\usepackage{amssymb,amsmath,amscd}
\usepackage{xcolor}

\DeclareFontFamily{U}{mathx}{\hyphenchar\font45}
\DeclareFontShape{U}{mathx}{m}{n}{ <5> <6> <7> <8> <9> <10>
   <10.95> <12> <14.4> <17.28> <20.74> <24.88> mathx10 }{}
\DeclareSymbolFont{mathx}{U}{mathx}{m}{n}
\DeclareFontSubstitution{U}{mathx}{m}{n}
\DeclareMathAccent{\widecheck}{0}{mathx}{"71}

\theoremstyle{plain}
\newtheorem{theorem}{Theorem}[section]
\newtheorem{prop}[theorem]{Proposition}
\newtheorem{lemma}[theorem]{Lemma}
\newtheorem{coro}[theorem]{Corollary}
\newtheorem{fact}[theorem]{Fact}

\theoremstyle{definition}
\newtheorem{definition}[theorem]{Definition}
\newtheorem{example}[theorem]{Example}
\newtheorem{remark}[theorem]{Remark}

\newcommand{\ts}{\hspace{0.5pt}}
\newcommand{\nts}{\hspace{-0.5pt}}

\newcommand{\CC}{\mathbb{C}\ts}
\newcommand{\RR}{\mathbb{R}\ts}
\newcommand{\ZZ}{{\ts \mathbb{Z}}}
\newcommand{\KK}{{\ts \mathbb{K} \ts }}
\newcommand{\SSS}{\mathbb{S}}
\newcommand{\TT}{\mathbb{T}}
\newcommand{\MM}{\mathbb{M}}
\newcommand{\NN}{\mathbb{N}}

\newcommand{\cA}{\mathcal{A}}
\newcommand{\cB}{\mathcal{B}}
\newcommand{\cM}{\mathcal{M}}
\newcommand{\cL}{\mathcal{L}}
\newcommand{\cO}{\mathcal{O}}
\newcommand{\cS}{\mathcal{S}}

\newcommand{\vG}{\varGamma}
\newcommand{\vL}{\varLambda}
\newcommand{\ii}{\mathrm{i}\ts}
\newcommand{\ee}{\mathrm{e}}
\newcommand{\dd}{\, \mathrm{d}}

\newcommand{\oplam}{\mbox{\Large $\curlywedge$}}
\newcommand{\exend}{\hfill $\Diamond$}
\newcommand{\defeq}{\mathrel{\mathop:}=}
\newcommand{\eqdef}{=\mathrel{\mathop:}}

\DeclareMathOperator{\dens}{dens}
\DeclareMathOperator{\udens}{\overline{dens}}
\DeclareMathOperator{\uudens}{\overline{u-dens}}
\DeclareMathOperator{\ldens}{\underline{dens}}
\DeclareMathOperator{\uldens}{\underline{u-dens}}
\DeclareMathOperator{\vol}{vol}

\DeclareMathOperator{\card}{card}
\DeclareMathOperator{\supp}{supp}

\newcommand{\Cc}{C_{\mathsf{c}}}
\newcommand{\Cz}{C^{}_{0}}

\newcommand{\WAP}{\mathcal{W\nts A\ts P}}
\newcommand{\SAP}{\mathcal{S\nts A \ts P}}
\newcommand{\myfrac}[2]{\frac{\raisebox{-2pt}{$#1$}}
      {\raisebox{0.5pt}{$#2$}}}

\begin{document}

\title[Doubly sparse measures]
{Pure point measures with sparse support\\[2mm]
 and sparse Fourier--Bohr support}

\author{Michael Baake}
\address{Fakult\"at f\"ur Mathematik, Universit\"at Bielefeld, \newline
\hspace*{\parindent}Postfach 100131, 33501 Bielefeld, Germany}
\email{mbaake@math.uni-bielefeld.de }

\author{Nicolae Strungaru}
\address{Department of Mathematical Sciences, MacEwan University, \newline
\hspace*{\parindent}10700 \ts 104 Avenue, Edmonton, AB, Canada T5J 4S2}
\email{strungarun@macewan.ca}

\author{Venta Terauds}
\address{Discipline of Mathematics, 
University of Tasmania, \newline
\hspace*{\parindent}Private Bag 37, Hobart, TAS 7001, Australia}
\email{venta.terauds@utas.edu.au}


\begin{abstract}
  Fourier-transformable Radon measures are called doubly sparse when
  both the measure and its transform are pure point measures with
  sparse support. Their structure is reasonably well understood in
  Euclidean space, based on the use of tempered distributions.  Here,
  we extend the theory to second countable, locally compact Abelian
  groups, where we can employ general cut and project schemes and the
  structure of weighted model combs, along with the theory of almost
  periodic measures. In particular, for measures with Meyer set
  support, we characterise sparseness of the Fourier--Bohr spectrum
  via conditions of crystallographic type, and derive representations
  of the measures in terms of trigonometric polynomials.  More
  generally, we analyse positive definite, doubly sparse measures in a
  natural cut and project setting, which results in a Poisson
  summation type formula.
\end{abstract}

\maketitle

\section{Introduction}

The study of translation-bounded, but possibly unbounded, measures on
a locally compact Abelian group (LCAG) $G$, with methods from harmonic
analysis, has a long history; compare \cite{ARMA1,BF,HR, ARMA}.  
{Of particular interest are Fourier-transformable measures 
$\mu$ such that both $\mu$ and $\widehat{\mu}$ are \emph{sparse},
which means that both are pure point measures and have locally finite
support.} The best-known example for this type of measure is the
uniform Dirac comb \cite{Cordoba} of a general lattice
$\vG\subset \RR^d$, which we write as
$\delta^{}_{\nts \vG} = \sum_{x\in\vG} \delta^{}_{x}$. This measure is
doubly sparse due to the Poisson summation formula (PSF),
\begin{equation}\label{eq:PSF}
     \widehat{\delta^{}_{\nts \vG}} \, = \, \dens (\vG) \,
     \delta^{}_{\nts \vG^0}  \ts ,
\end{equation}
where $\vG^0$ denotes the dual lattice of $\vG$; see
\cite[Sec.~9.2]{TAO1} and references therein for background.

The understanding of such measures, and translation-bounded measures
and their transforms in general, has reached a reasonably mature state
for $G=\RR^d$, where they arise in the study of quasicrystals. Here,
Meyer's pioneering work on model sets \cite{Mey72,Mey94} plays a key
role; see \cite{Moo97, Moo00, BLM} for a detailed account, and
\cite[Ch.~9]{TAO1} for an exposition of their appearance in
diffraction theory.  {Though model sets typically lead to
  diffraction measures with dense support, the methods from this field
  provide immensely useful tools for the questions at hand.}
{In particular, we will be able to classify, in
  Theorem~\ref{thm:trig-coeffs}, the few cases of
  Fourier-transformable measures that are supported on cut and project
  sets and have a sparse Fourier transform.}  While the natural
setting of tempered distributions simplifies the harmonic analysis in
this case significantly, {and powerful complex-analytic
  techniques may be applied,} several interesting open problems
remain.  We particularly mention those collected and stated by
Lagarias \cite{Jeff-rev}, some of which have recently been answered by
{Kellendonk and Lenz \cite{KL}, by Favorov \cite{Fav}, and
  by} Lev and Olevskii \cite{LO1,LO2}.

In this paper, we substantially extend the setting and consider
\emph{doubly sparse} measures on an LCAG $G$ that is also second
countable, hence $\sigma$-compact and metrisable.  By a doubly sparse
measure we mean a Fourier-transformable measure $\mu$ such that both
$\supp (\mu)$ and $\supp (\widehat{\mu})$ are locally finite point
sets {(satisfying an upper density condition as detailed in
  Section~\ref{sec:sparse})} in $G$ and $\widehat{G}$,
respectively. In particular, both $\mu$ and $\widehat{\mu}$ must be
pure point measures. Beyond the lattice Dirac comb in \eqref{eq:PSF},
{other notions and examples of doubly sparse measures have
  been studied} in \cite{LO1,LO2,Meyer} under the name `crystalline
measures'. We do not adopt this term because it has a
different meaning elsewhere. {Note that some of the
  measures appearing in these papers are not doubly sparse in our
  sense, and do not seem to be compatible with the cut and project
  formalism, which makes them unsuitable for our tools.}

In this wider generality, we can no longer work with tempered
distributions, but need {an extension that is suitable
for LCAGs. While one option could be Bruhat--Schwartz theory, 
compare \cite{Osb} and references therein, it seems more natural to us}
to employ the general theory of Radon measures
on locally compact Abelian groups. {A
large body of results on such measures has
accumulated in recent years, due to the systematic development of
the theory of aperiodic order, including the cut and project scheme
for measures and their Fourier transforms. We will make extensive
use of some of the recent results; these, to our knowledge, have
no counterpart yet in Bruhat--Schwartz space. Moreover,
we shall employ the connection between Fourier transform and almost 
periodicity for measures and tempered distributions \cite{ST}.}
Since the measures under consideration need not be finite, the notion of
transformability is non-trivial; see \cite[Ch.~4.9]{TAO2} for a
detailed exposition and \cite{ARMA, BF,HR} for background.

{The measures of interest will often display a high degree of
long-range translational order. Thus,} we may profit from the methods
developed in \cite{BM}, which have recently been systematised and
extended in \cite{NS11}.  In particular, we shall need almost periodic
measures of various kinds that emerge from a \emph{cut and project
  scheme} (CPS) in the sense that they are supported on a projection
set with certain properties; see \cite[Ch.~7]{TAO1} for an
introduction, and \cite{Mey72,Moo97,Moo00} for the general theory and
more advanced topics.  

{In the particular case $G = \RR^d$, a natural question is whether 
one could get more general results via the theory of tempered distributions. 
It turns out that for a large class of measures, which includes the typical 
examples we are interested in, the Fourier theory of Radon measures 
and that of tempered distributions coincide; see Lemma~\ref{lem:FT-able}
for details.}  \vspace*{3mm}

The paper is organised as follows. We recall various concepts and
preliminaries in Section~\ref{sec:prelim}, followed by
Section~\ref{sec:sparse} on the notion and basic properties of sparse
point sets in LCAGs. Then, we look more closely at Radon measures with
Meyer set support in Section~\ref{sec:meyer}, which contains two of
our central results, namely Theorems~\ref{thm:conj1} and
\ref{thm:trig-coeffs}. They assert that such measures exhibit the
following dichotomy: Either $\mu$ and $\widehat{\mu}$ are supported on
fully periodic sets, or $\widehat{\mu}$ meets the translates of any
open set in unboundedly many points.

Then, in Section~\ref{sec:sap}, we consider positive definite measures
with uniformly discrete support and sparse Fourier--Bohr spectrum.  In
particular, we show that any such measure is norm-almost periodic and
thus permits a representation in a natural cut and project scheme; see
Theorem~\ref{thm:T2} and Corollary~\ref{coro:posdef}. This also allows
us to express $\widehat{\mu}$ in terms of a PSF-type formula and to
discuss the connection with diffraction theory.  Finally, in
Section~\ref{sec:real}, we put our results in perspective with
previous results of Lev and Olevskii \cite{LO1,LO2} by considering
measures supported on $\RR^d$, including those arising from fully
Euclidean cut and project schemes.

\section{Notation and Preliminaries}\label{sec:prelim}

Below, we use the general setting of the monograph \cite{TAO1}, and
refer to \cite[Chs.~4 and 5]{TAO2} for background on the Fourier
theory of Radon measures on LCAGs. {From now on,
unless stated otherwise, the term `measure' will refer
to a (generally complex) Radon measure.}

We assume an LCAG $G$ to be equipped with its Haar measure
$\theta^{}_{G}$ in a suitable normalisation.  This means that we
arrange $\theta^{}_{G}$ and $\theta^{}_{\widehat{G}}$, where
$\widehat{G}$ is the Pontryagin dual of $G$, relative to each other in
such a way that Parseval's equation holds. In particular, we shall use
Lebesgue measure on $\RR^d$ and counting measure on $\ZZ^m$, while the
Haar measure will usually be normalised for compact groups. As a
consequence, the Haar measure on a finite discrete group will be
counting measure divided by the order of the group. For a measurable
set $A \subseteq G$, we will often write $\vol (A)$ instead of
$\theta^{}_{G} (A)$ and $\dd x$ as a shorthand for
$\dd \theta^{}_{G} (x)$, if the reference to $G$ is
unambiguous. Below, we will be concerned with certain point sets in
$G$, where the term \emph{point set} refers to an at most countable
union of singleton sets.

When $G$ is an LCAG and $g\in L^1 (G)$, we write the \emph{Fourier
  transform} of $g$ as
\[
     \widehat{g} (\chi) \, = \int_{G}
     \overline{\chi(x)} \, g(x) \dd x \ts ,
\]
where $\chi \in \widehat{G}$ is a continuous character, with
$\overline{\chi} = \chi^{-1}$.  Likewise, the matching inverse
transform is given by
$\widecheck{g} (\chi) = \int_{G} \chi (x) \ts g(x) \dd x$.  In this
formulation, $\widehat{G}$ is written multiplicatively. This has to be
compared with the widely used additive notation for $G=\RR^d$, where
one writes $\chi (x) = \chi^{}_{k} (x) = \ee^{2 \pi \ii k x}$ with
$k\in\RR^d$. Here, and in similar situations such as the $d$-torus, we
then write
$\widehat{g} (k) = \int_{G} \overline{\chi^{}_{k} (x)} \ts g(x) \dd x$
with $k\in \widehat{G}$, now written additively. {From here, we take
the usual route to define the Fourier transform of finite measures,
and the notion of Fourier transformability of Radon measures, as in
\cite[Def.~4.9.7]{NS11}.}

A \emph{van Hove sequence} $\cA = \{ A_n \}$ in $G$ is a sequence of
compact sets $A_n \subseteq G$ that are nested and exhaustive, meaning
$A^{}_n \subseteq A^{\circ}_{n+1}$ together with $\bigcup_n A_n = G$,
and also satisfy the asymptotic condition
\[
      \lim_{n\to\infty}  \frac{\theta^{}_{G} (\partial^K \! A_n)}
      {\theta^{}_{G} (A_n)} \, = \, 0 
\]
for any compact $K\subseteq G$. Here, for compact $K$ and $A$, the
\emph{$K$-boundary of $A$} is defined as
\begin{equation}\label{eq:K-boundary}
    \partial^K \! A \, \defeq \, \bigl( \ts 
    \overline{ (A+K) \setminus A} \ts \bigr)
    \cup \bigl( A \cap ( \ts \overline{G\setminus A}
      -K \ts) \bigr) ,
\end{equation}
where $A\pm K \defeq \{ a \pm k : a \in A, k \in K \}$ denotes the
Minkowski sum and difference of the two sets $A$ and $K$. In
particular, for all compact $K\subseteq G$, one has
\begin{equation}\label{eq:K-bound-inclusion}
A+K \subseteq A \cup \partial^K \! A\ts .
\end{equation}
The nestedness condition implies that $\bigcup_n A^{\circ}_{n+1}$ is
an open cover of $G$, and hence of any compact set $K\subseteq G$.
Consequently, $K\subseteq \bigcup_{n\in F} A^{\circ}_{n+1}$ for some
finite set $F \subset \NN$, which means $K\subseteq A_m$ for all
sufficiently large $m$.

Note that van Hove sequences of the type defined here do exist in all
$\sigma$-compact LCAGs; see \cite[p.~145]{Martin2}.  In fact, since we
included nestedness and exhaustion of $G$ into our definition of a van
Hove sequence, the existence of such sequences becomes equivalent to
$\sigma$-compactness of $G$. One can go beyond this situation, but we
do not attempt that here.

For the induced continuous translation action of $G$ on functions and
measures, we start from the relation
$\bigl( T^{}_{t} \ts g \bigr) (x) = g (x-t)$ for functions. The
matching definition for measures is
\[
    \bigl( T^{}_{t} \ts \mu \bigr) (g) \, = \,
    \mu ( T^{}_{-t} \, g )
\]
for test functions $g \in C_{\mathsf{c}} (G)$. The convolution is
defined as usual, and one checks that
\begin{equation}\label{eq:conv-1}
   (T^{}_{t} \ts \mu ) * g \, = \, T^{}_{t} (\mu*g) \ts ,
\end{equation}
which makes the notation $T^{}_{t} \ts\mu*g$ unambiguous.
In particular, one finds
\begin{equation}\label{eq:conv-2}
   \bigl( T^{}_{t} \ts \mu * g \bigr) (y)
   \, = \, \bigl( \mu*g \bigr) (y-t) \ts .
\end{equation}

Let $G$ be a fixed LCAG.  Recall that a measure $\mu$ on $G$ is called
\emph{translation bounded} {if
\begin{equation}\label{eq:def-norm}
  \| \mu \|^{}_{E} \, \defeq \,
  \sup_{x\in G} \lvert \mu \rvert (x+E) \, < \, \infty
\end{equation}
holds} for any compact set $E$. One can equivalently demand
that $\mu \ast g$ be a bounded function for all
$g \in C_{\mathsf{c}} (G)$; see \cite[Sec.~1]{Martin2} for the case
that $G$ is $\sigma$-compact, and \cite[Thm.~1.1]{ARMA1} as well as
\cite[Prop.~4.9.21]{MoSt} for the general case.  We denote the set of
translation-bounded measures by $\cM^{\infty} (G)$, which will show up
many times below.

\section{Sparse sets}\label{sec:sparse}

For the remainder of the paper, unless stated otherwise, $G$ will
stand for a second-countable LCAG, and $\widehat{G}$ for its dual
group. {We generally need second countability of $G$
  to define doubly sparse measures on $G$, and will 
explicitly mention when our setting can be extended.}
Recall that a topological group $G$ is \emph{second-countable} 
if there exists a countable basis for its topology. A second countable
group $G$ is both $\sigma$-compact and metrisable, which means that
$\widehat{G}$ has the same properties \cite[Thm.~4.2.7]{Reiter}.

If $\mu$ is a transformable measure on $G$, we call the measurable
support of $\widehat{\mu}$ the \emph{Fourier--Bohr support} of $\mu$,
and abbreviate it as FBS from now on. In some papers
\cite{LO1,LO2,Meyer}, the FBS is also called the \emph{spectrum} or
the Fourier--Bohr spectrum of $\mu$. Below, we will not adopt this
terminology because the term \emph{spectrum} is already in use in
several ways in related questions from dynamical systems and ergodic
theory.

\subsection{General notions and properties}

Given a point set $\vL \subseteq G$ and a van Hove sequence
$\cA=\{ A_n \}$ in $G$, we define the \emph{upper density} and the
\emph{uniform upper density} of $\vL$ with respect to $\cA$ to be
\[
\begin{split}
   \udens^{}_{\cA}(\vL) \, & \defeq \,
    \limsup_{n \to \infty} \, 
    \frac{\card ( \vL \cap A_n )} {\vol (A_n)} 
    \qquad \text{and} \\[2mm]
    \uudens^{}_{\cA}(\vL) \, & \defeq \,
    \limsup_{n \to \infty} \, \sup_{x \in G} 
    \frac{\card \bigl( \vL \cap (x+ \nts A_n)\bigr)}
         {\vol (A_n)} \ts ,
    \end{split}
\]
respectively, and similarly for the lower densities, then denoted as
$\ldens^{}_{\cA} (\vL)$ and $\uldens^{}_{\cA} (\vL)$, with $\limsup$
and $\sup$ replaced by $\liminf$ and $\inf$, respectively. When the
lower density of a point set $\vL$ agrees with its upper density, the
\emph{density} of $\vL$ with respect to $\cA$ exists, and is denoted
as $\dens^{}_{\cA} (\vL)$.  The \emph{total uniform upper density}
refers to
\[
   \uudens(\vL) \, \defeq \, \sup \big\{ \uudens^{}_{\cA}(\vL)
    : \cA \mbox{ is a van Hove sequence} \big\} ,
\]
again with the matching definition for $\uldens (\vL)$.

Let us add a comment on these notions.  When a point set $\vL$ has a
finite uniform upper density with respect to \emph{some} van Hove
sequence $\cA$, it actually has finite uniform upper density with
respect to \emph{all} van Hove sequences and, furthermore, the
supremum over all of these is finite; see Lemma~\ref{lem:L1} and
Remark~\ref{rem:not-ud} below.  In contrast, a point set may have
finite upper density with respect to some van Hove sequence, but
infinite upper density with respect to another; see
Example~\ref{ex:wud-sparse}. For this reason, we do not consider the
concept of total upper density, and we define sparseness with respect
to a particular van Hove sequence in $G$.

{The uniform density is sometimes called 
\emph{upper Banach density}. When $G$ is a discrete LCAG, this
density does not depend on the choice of the F{\o}lner sequence
\cite{DHZ}. One thus has $\uudens ( \vL) = \uudens^{}_{\cA} (\vL)
\leqslant 1$ for all $\vL$ and every F{\o}lner sequence $\cA$ in $G$.
The situation seems to be more complicated in non-discrete groups.}

\begin{definition}\label{def:sparse}
  Given a van Hove sequence $\cA = \{ A_n \}$ in $G$, a point set
  $\vL \subseteq G$ is called \emph{$\cA$-sparse} if
  $\udens^{}_{\cA}(\vL) < \infty$, and \emph{strongly $\cA$-sparse} if
  $\uudens^{}_{\cA} (\vL) < \infty$.  Moreover, $\vL$ is
  \emph{strongly sparse} if it is strongly $\cA$-sparse for every van
  Hove sequence $\cA$ in $G$.
\end{definition}

\begin{remark}\label{rem:loc-fin}
  If a point set $\vL\subseteq G$ is $\cA$-sparse for some van Hove
  sequence $\cA = \{ A_n \}$ in $G$, it is automatically locally
  finite.  Indeed, if $K\subseteq G$ is any compact set, there is some
  $A_n$ in $\cA$ with $K\subseteq A_n$, and one has
\[
    \card (\vL \cap K ) \, \leqslant \,
    \card (\vL \cap A_n ) \, < \, \infty 
\]   
due to $\cA$-sparseness. Local finiteness of $\vL$ is then clear,
which equivalently means that $\vL$ is discrete and closed; compare
\cite[Sec.~2.1]{TAO1}. \exend
\end{remark}

Next, we need to recall a notion that is slightly weaker than uniform
discreteness, {where a point set $\vL\in G$
is \emph{uniformly discrete} if some open neighbourhood $U$ of $0$ in $G$
exists such that, for any two distinct points $x,y\in \vL$, one has
$(x+U) \cap (y+U) = \varnothing$.}

\begin{definition}\label{def:wud}
  A point set $\vL\subseteq G$ is called \emph{weakly uniformly
    discrete} if, for each compact $K\subseteq G$ and all $x \in G$,
  $\card \bigl(\vL \cap ( x + \nts K \ts )\bigr)$ is bounded by a 
  constant that depends only on $K$.
\end{definition}
  
Weak uniform discreteness of $\vL$ is equivalent to
$\delta^{}_{\!\vL}$ being a translation-bounded measure; compare
\cite[p.~288]{NS11} as well as \cite[Sec.~1]{Martin2}.  Note also that
strong $\cA$-sparseness clearly implies $\cA$-sparseness, but not vice
versa. {Let us illustrate these connections as follows.}

\begin{example}\label{ex:wud-sparse}
  Consider the point set $\vL \subset \RR$ defined as
\[
  \vL \, = \, \bigcup_{n \in \NN}  \big\{  n + \myfrac{k}{n} : 
      0 \leqslant k < n \big\} .
\] 
The set $\vL$ fails to be weakly uniformly discrete because
$\card \bigl( \vL \cap (n + [0,1])\bigr) =n$ is unbounded.  For the
same reason, $\vL$ cannot be strongly $\cA$-sparse, as any van Hove
sequence $\cA=\{ A_n \}$ in $\RR$ has the property that the compact
sets $A_n$ contain a translate of $[0,1]$ for all sufficiently large
$n$, so $\uudens^{}_{\cA} (\vL) = \infty$, and thus also
$\uudens (\vL) = \infty$.

However, $\vL$ can still be $\cA$-sparse for certain van Hove
sequences. In general, the density with respect to a given van Hove
sequence need not be zero, but can take any value $\geqslant 0$, even
including $\infty$.  Indeed, choosing $A_n$ as $[-n^3,n]$,
$[-\alpha n^2,n]$ with $\alpha > 0$, or $[-n,n^2]$, one gets
$\cA$-density $0$, $\frac{1}{2\ts \alpha}$, or $\infty$, respectively.
\exend
\end{example}

\begin{lemma}\label{lem:L1} 
  If $\vL \subseteq G$ is weakly uniformly discrete, one has
\[
    \sup  \big\{  \udens^{}_{\cA} (\vL) : \cA
    \text{ is van Hove in } G \big\} \, \leqslant \:
    \uudens (\vL) \, < \, \infty \ts .
\]   
\end{lemma}

  {In \cite[Lemma 9.2]{LR1}, the authors prove
    this result for the larger class of translation-bounded
    measures (compare also with \cite[Lemma~1.1]{Martin2}).
    Here, we prefer to give an independent argument as follows.}
\begin{proof}
    Observe first that
  $\udens^{}_{\cA} (\vL) \leqslant \uudens^{}_{\cA} (\vL)$ obviously
  holds for any van Hove sequence $\cA$ in $G$, hence also
  $\udens^{}_{\cA} (\vL) \leqslant \uudens (\vL)$ for all $\cA$, and
  the first inequality is clear.  It remains to show that there is a
  constant $C<\infty$ with $\uudens (\vL) \leqslant C$.
  
  Select some non-negative $f \in C_{\mathsf{c}} (G)$ with
  $\theta^{}_{G} (f) = \int_G f(x) \dd x = 1$, and set
  $K = \supp (f)$.  Since $\vL$ is weakly uniformly discrete, the
  Dirac comb $\delta^{}_{\! \vL}$ is translation bounded, and
  $f \nts * \delta^{}_{\! \vL}$ is a non-negative continuous function
  that is bounded. We thus have
  $C \defeq \| f \nts * \delta^{}_{\! \vL} \|^{}_{\infty} < \infty$
  and
  $0 \leqslant \bigl( f \nts * \delta^{}_{\! \vL} \bigr) (x) \leqslant
  C$ for all $x\in G$.
  
  Let $\cA$ be any van Hove sequence in $G$. Then, using 
  Fubini, we can estimate
\[
\begin{split}
    \card & \bigl( \vL \cap (x+A_n ) \bigr) \,  = \int_G \int_G 
    f(t) \dd t \, 1^{}_{x+A_n} (s) \dd \delta^{}_{\! \vL} (s) \\[2mm]
    & = \int_G \int_G f(t-s) \dd t \, 1^{}_{x+A_n} (s) 
    \dd \delta^{}_{\! \vL} (s) \, = \int_G \int_G f(t-s) \, 
    1^{}_{x+A_n} (s) \dd \delta^{}_{\! \vL} (s) \dd t \ts .
\end{split}
\]  
  Now, observe that $f(t-s) \, 1^{}_{x+A_n} (s) =0$ unless
  $t\in x+A_n +K$, hence
\[
   0 \, \leqslant \, f(t-s) \, 1^{}_{x+A_n} (s) \, = \, 
   f(t-s) \, 1^{}_{x+A_n} (s) \, 1^{}_{x+A_n +K} (t)
   \, \leqslant \, f(t-s) \, 1^{}_{x+A_n +K} (t) \ts ,
\]  
and we get
\begin{equation}\label{eq:last-step}
\begin{split}
    \card & \bigl( \vL \cap (x+A_n ) \bigr) \,  \leqslant
    \int_G \int_G f(t-s) \, 1^{}_{x+A_n +K} (t) 
    \dd \delta^{}_{\! \vL} (s) \dd t \\[2mm]
    &=  \int_G 1^{}_{x+A_n +K} (t) \int_G f(t-s) 
    \dd \delta^{}_{\! \vL} (s) \dd t
    \, = \int_G 1^{}_{x+A_n +K} (t) \, \bigl(f*\delta^{}_{\! \vL} 
    \bigr) (t) \dd t \\[4mm]
    & \leqslant \, C \, \vol(x+A_n +K) \, = \, C \, \vol(A_n +K)
    \, \leqslant \, C \bigl( \vol (A_n) + 
    \vol (\partial^K \! A_n) \bigr) ,
\end{split}
\end{equation}
independently of $x$, with the last step following from
Eq.~\eqref{eq:K-bound-inclusion}. Consequently, we have
\[
     \sup_{x\in G} \frac{\card \bigl( \vL \cap (x+A_n)\bigr)}
     {\vol (A_n)} \, \leqslant \, C \left( 1 +
     \frac{\vol (\partial^K \! A_n)}{\vol (A_n)} \right) ,
\]
where $n$ is arbitrary. Hence, by the van Hove property,
$\uudens^{}_{\cA} (\vL) \leqslant C$. Since this bound does not depend
on $\cA$, our claim follows.
\end{proof}

\begin{remark}\label{rem:not-ud} 
  {When $\vL \subseteq G$ is a point set that violates weak
    uniform discreteness, one gets $\uudens_{\cA}(\vL)=\infty$, for
    any van Hove sequence $\cA$.  Indeed,} the sets $A_n$ are compact,
  and we may, without loss of generality, assume that all of them have
  non-empty interior. For any $n\in\NN$, this implies
\[
      \| \delta^{}_{\! \vL} \|^{}_{A_n} \, = \;
      \sup_{x\in G} \, \lvert \delta^{}_{\! \vL} \rvert (x + A_n)
      \, = \, \infty \ts ,
\]
which really is a statement in the norm topology \cite{BM};
compare \cite[Eq.~(5.3.1)]{NS11}. This property means that 
\[ 
    \sup_{x \in G}
    \frac{\card \bigl( \vL \cap (x+ \nts A_n)\bigr)} 
    {\vol (A_n)} \, = \, \infty 
\]
and hence $ \uudens^{}_{\cA}(\vL) = \infty$.    \exend
\end{remark}

{Under the conditions of 
Remark~\ref{rem:not-ud}, for any van Hove sequence $\cA$,
there is a sequence $\{ t_n \}$ of translations such that
$\card \bigl( \vL \cap (t_n + A_n ) \bigr) /\vol (A_n) > n$, which is
unbounded. However, this does not imply $ \udens^{}_{\cA}(\vL) =
\infty$, as Example~\ref{ex:wud-sparse} shows.}
Also, Lemma~\ref{lem:L1} and Remark~\ref{rem:not-ud} imply the
following: If $\uudens^{}_{\cA}(\vL)< \infty$ holds for \emph{some} 
van Hove sequence $\cA$, the same estimate holds for \emph{all} 
van Hove sequences. We can now strengthen the relations as follows.

\begin{theorem}\label{thm:equivalences}
  For a point set\/ $\vL\subseteq G$, the following properties are
  equivalent.
\begin{enumerate}\itemsep=2pt
   \item $\vL$ is weakly uniformly discrete.
   \item $\vL$ is strongly sparse.
   \item One has\/ $\uudens (\vL) < \infty$.
   \item One has\/ $\uudens^{}_{\nts\cA} (\vL) < \infty$ for some 
       van Hove sequence\/ $\cA$.
\end{enumerate}   
\end{theorem}

\begin{proof}
  (1) $\Rightarrow$ (3) follows from Lemma~\ref{lem:L1}, while (3)
  $\Rightarrow$ (2) $\Rightarrow$ (4) is an immediate consequence of
  Definition~\ref{def:sparse}. Finally, (4)
  $\Rightarrow$ (1) follows from Remark~\ref{rem:not-ud}.
\end{proof}

\subsection{Sparse cut and project sets}\label{sec:CPS}

Let us begin by briefly recalling the setting of a cut and project
scheme (CPS), which is based on \cite{Mey72,Moo97, Moo00}. A CPS
consists of two LCAGs, $G$ and $H$, together with a 
lattice\footnote{In an LCAG $G$, a lattice simply is a
discrete, co-compact subgroup.}
$\cL \subseteq G \nts\nts \times \! H$ and several mappings with
some specific conditions. This is denoted by the triple $(G, H, \cL)$
and usually summarised in a diagram as follows.
\begin{equation}\label{eq:CPS-1}
\renewcommand{\arraystretch}{1.2}\begin{array}{r@{}ccccc@{}l}
   & G & \xleftarrow{\:\; \pi^{}_{G} \;\: } 
      & G \nts\nts \times \nts\nts H & 
        \xrightarrow{\;\: \pi^{}_{H} \;\: } & H & \\
   & \cup & & \cup & & \cup & \hspace*{-1ex} 
   \raisebox{1pt}{\text{\footnotesize dense}} \\
   & \pi^{}_{G} (\cL) & \xleftarrow{\;\ts 1-1 \;\ts } & \cL & 
        \xrightarrow{ \qquad } &\pi^{}_{H} (\cL) & \\
   & \| & & & & \| & \\
   & L & \multicolumn{3}{c}{\xrightarrow{\qquad\qquad\;\,\star
       \,\;\qquad\qquad}} 
       &  {L_{}}^{\star\nts}  & \\
\end{array}\renewcommand{\arraystretch}{1}
\end{equation}
Here, the mapping
$(\cdot)^{\star} \! : \, L \xrightarrow{\quad} H$ is well defined; see
\cite{Moo97,Moo00} for a general exposition and \cite{TAO1} for
further details, in particular for the case of $G=\RR^d$, which we
call a \emph{Euclidean} CPS.  When also $H=\RR^n$, it is called
\emph{fully Euclidean}.

For some arguments, we also need the \emph{dual CPS}, denoted by
$(\widehat{G}, \widehat{H}, \cL^{0})$ and nicely explained in
\cite{Moo97}; see also \cite{Schr}. Here, $\widehat{G}$ and
$\widehat{H}$ are the dual groups, while $\cL^{0}$ is the annihilator
of $\cL$ from \eqref{eq:CPS-1}, and a lattice in
$\widehat{G\nts\nts \times \! H} \simeq \widehat{G} \nts\nts
\times \! \widehat{H}$. Diagrammatically, we get the following.
\begin{equation}\label{eq:CPS-2}
\renewcommand{\arraystretch}{1.2}\begin{array}{r@{}ccccc@{}l}
      & \widehat{G}
      & \xleftarrow{\:\; \pi^{}_{\widehat{G}} \;\: } 
      & \widehat{G} \nts\nts \times \! \widehat{H} & 
         \xrightarrow{\;\: \pi^{}_{\widehat{H}} \;\: }
      & \widehat{H} & \\
   & \cup & & \cup & & \cup & \hspace*{-1ex} 
   \raisebox{1pt}{\text{\footnotesize dense}} \\
    & \pi^{}_{\widehat{G}} (\cL^{0})
    & \xleftarrow{\;\ts 1-1 \;\ts } & \cL^{0} & 
        \xrightarrow{ \qquad } &\pi^{}_{\widehat{H}} (\cL^{0}) & \\
   & \| & & & & \| & \\
   & L^{0} & \multicolumn{3}{c}{\xrightarrow{\qquad\qquad\;\,\star
       \,\;\qquad\qquad}} 
       &  \bigl( L_{}^{0} \bigr)^{\star\nts}  & \\
\end{array}\renewcommand{\arraystretch}{1}
\end{equation}
Note that the existence of a $\star$-map in the dual CPS follows from
that in the original one, whence we use the same symbol for it, though
the mappings are, of course, different.

Recall that, once a CPS $(G, H, \cL)$ with its natural projections and
its $\star$-map is given, a \emph{cut and project set} is a set of the
form
\begin{equation}\label{eq:CPS-set}
   \oplam (U)  \, = \,
   \{ x\in \pi^{}_{G} (\cL) : x^{\star} \in U \}
   \, = \, \{ x\in L : x^{\star} \in U \}
\end{equation} 
for some coding set or \emph{window} $U \subseteq H$. When $U$ is
relatively compact with non-empty interior, $\oplam (U)$ is called a
\emph{model set}. Note that model sets are Meyer 
sets,\footnote{Recall that $\vL\subseteq G$ is a Meyer set if
it is relatively dense and if $\vL-\vL \subseteq \vL +F$ holds for
some finite set $F\subseteq G$. Another characterisation together
with further aspects will be discussed in Remark~\ref{rem:harvest}.} 
and that any Meyer set is a subset of a model set; see 
\cite[Thm.~5.7.8]{MoSt}. For a function $g$ on $H$ such that
\[
     \omega^{}_{g} \; \defeq \!\! \sum_{x\in\pi^{}_{G}(\cL)} 
       \!\! g (x^{\star}) \, \delta^{}_{x} 
\]
is a measure on $G$, we call $\omega_g$ a \emph{weighted Dirac comb}
for $(G, H, \cL)$; see \cite[Sec.~4.1]{CRS} for details. When the
support of $\omega^{}_{g}$ is a model set, we call it a \emph{weighted
  model comb}.

Recall that the density of a lattice, such as $\cL$ in
$G \nts\nts \times \! H$, exists uniformly, so does not depend on the
choice of a van Hove sequence.  We thus write $\dens (\cL)$ in this
situation.  Let us begin by proving a density formula for cut and
project sets with open sets as windows, which will be a key input for
many of our later computations. Here, we invoke and extend
\cite[Prop.~3.4]{HuRi}, which is a density formula for relatively
compact sets as windows {that is substantially based on
  \cite[Thm.~1]{Martin1}. We note that, while the point sets we employ
  in our results often fail to be model sets themselves, projection
  sets with unbounded windows of finite measure will play an important
  role in our arguments. Also, it is essential for the proofs to come
  that the windows need not be regular. There is quite some recent
  interest in the corresponding theory of weak model sets; compare
  \cite{Schr,BHS,KR2,KR1}.}

\begin{prop}\label{prop:sparse-CPS} 
  Let\/ $(G, H, \cL)$ be a CPS, let\/ $\cA$ be a van Hove sequence
  in\/ $G$, and let\/ $U \subseteq H$ be an open set. Then,
\[
   \theta^{}_{\nts H} (U) \, \leqslant \,  
   \frac{\udens^{}_{\nts\cA} \bigl(\oplam(U)\bigr)}
   {\dens (\cL)} \ts .
\]
In particular, if $\oplam(U)$ is\/ $\cA$-sparse, one has\/
$\theta^{}_{\nts H} (U) <\infty$.  
\end{prop}

\begin{proof}
  Let $K \subseteq U$ be any compact set. Then, by
  \cite[Prop.~3.4]{HuRi}, we have
\[
     \theta^{}_{\nts H} (K) \, \leqslant \,
     \frac{\udens^{}_{\nts\cA} \bigl(\oplam(K)\bigr)}
     {\dens (\cL)} \ts .
\]
Next, as $K \subseteq U$, we have $\oplam(K) \subseteq \oplam(U)$ and
hence
$\udens^{}_{\nts\cA} \bigl(\oplam(K)\bigr) \leqslant
\udens^{}_{\nts\cA} \bigl(\oplam(U)\bigr)$.  This shows that, for all
$K \subseteq U$ compact, we have
\[
   \theta^{}_{\nts H} (K) \, \leqslant \,
   \frac{\udens^{}_{\nts\cA} \bigl(\oplam(U)\bigr)}
   {\dens (\cL)} \ts .
\]
Finally, by the inner regularity of $\theta^{}_H$, we have 
\[
  \theta^{}_{\nts H} (U) \, =
  \sup_{\substack{K \subseteq U \\ \text{compact}}}
  \! \theta^{}_{\nts H} (K) \, \leqslant \,
  \frac{\udens^{}_{\nts\cA} \bigl(\oplam(U)\bigr)}
  {\dens (\cL)} \ts ,
\]
which completes the argument.
\end{proof}

\begin{remark} 
  It is worth mentioning that, given a relatively compact window
  $W \subseteq H$ and an arbitrary van Hove sequence $\cA$ in $G$, one
  has the following chain of estimates,
\[
\begin{split}
    \dens (\cL) \, \theta^{}_{H} \bigl(W^{\circ}\bigr) \, & \leqslant \,
    \uldens \bigl( \oplam (W) \bigr) \, \leqslant \,
    \uldens^{}_{\cA} \bigl( \oplam (W)\bigr) \\[1mm]
  & \leqslant \, \ldens^{}_{\cA} \bigl( \oplam (W)\bigr)
     \, \leqslant \, \udens^{}_{\nts\cA} \bigl( \oplam (W)\bigr) \\[1mm]
  & \leqslant \, \uudens^{}_{\nts\cA}
     \bigl( \oplam (W)\bigr) \, \leqslant \,
    \uudens \bigl( \oplam (W)\bigr) \, \leqslant \,
    \dens (\cL) \, \theta^{}_{H} 
    \bigl(\ts \overline{W}\ts \bigr) \ts ,  
\end{split}
\]  
which puts Proposition~\ref{prop:sparse-CPS} in a more general
perspective.  \exend
\end{remark}

An immediate consequence of Proposition~\ref{prop:sparse-CPS} 
is the following.

\begin{coro}\label{coro:cor1} 
  Let\/ $(G, H, \cL)$ be a CPS, let\/ $h \in \Cz(H)$ and set\/
  $U\defeq \{ z\in H : h(z) \neq 0 \}$. If the weighted Dirac comb\/
  $\omega^{}_h$ has\/ $\cA$-sparse support for some van Hove
  sequence\/ $\cA$ in $G$, one has\/ $\theta^{}_{\nts H} (U) <\infty$.
  In particular, $\theta^{}_{\nts H} (U) <\infty$ whenever\/
  $\supp(\omega^{}_h)$ is weakly uniformly discrete.
\end{coro}

\begin{proof}
 Observe that
 \[
    \supp(\omega^{}_{h}) \, = \,
     \{ x\in L : h(x^{\star}) \neq 0 \} 
     \, = \, \{ x\in L : x^{\star}\in U \}
     \, = \, \oplam(U) \ts .
 \]
 Then, $\cA$-sparseness of the support means that
 $\udens^{}_{\cA} \bigl( \oplam (U)\bigr)$ is finite, and the result
 follows from Proposition~\ref{prop:sparse-CPS}.
 
 {The last claim now follows via the implication
 (1) $\Rightarrow$ (4) from Theorem~\ref{thm:equivalences}.}
\end{proof}

To continue, we will have to consider a group $G$ and its dual,
$\widehat{G}$.  Unless stated otherwise, we assume that we have
selected a van Hove sequence in each of these two groups, namely $\cA$
for $G$ and $\cB$ for $\widehat{G}$.  In the case of a self-dual
group, such as $\RR^d$, we might think of taking the same sequence for
both. In contrast, for $G=\ZZ^{m}$ hence $\widehat{G} = \TT^{m}$, we
fix some $\cA$ for $\ZZ^m$, say a sequence of centred cubes or balls,
while it would be natural to take $\cB = \{ B_n \}$ as the constant
sequence, so $B_n = \TT^{m}$ for all $n$, and similarly for other
LCAGs $H$ that are compact. Note that this is consistent with our
nestedness condition because $H$ is both open and closed.

\begin{definition} 
  Let $G$ be an LCAG with $\sigma$-compact dual group,
  $\widehat{G}$. Assume that a van Hove sequence $\cB$ for
  $\widehat{G}$ has been selected.  Then, we say that a measure
  $\mu \in \cM^\infty(G)$ has \emph{sparse Fourier--Bohr support}
  (FBS) with respect to $\cB$ if
\begin{enumerate}\itemsep=2pt
\item $\mu$ is Fourier transformable, with transform
  $\widehat{\mu}\ts$;
\item the support, $\supp(\widehat{\mu})$, is a
  $\cB$-sparse point set in $\widehat{G}$.
\end{enumerate}
Moreover, if also $G$ is $\sigma$-compact and a van Hove sequence
$\cA$ for $G$ is given, a measure $\mu$ with $\cA$-sparse support and
$\cB$-sparse FBS is called \emph{doubly sparse} with respect to
$(\cA, \cB \ts )$, or $(\cA, \cB \ts)$-sparse for short.  If $\mu$ is
$(\cA, \cB \ts)$-sparse for any pair of van Hove sequences, we simply
call $\mu$ \emph{doubly sparse}.
\end{definition}

\begin{remark}\label{rem:sparse}
  Note that the notion of a $\cB$-sparse FBS does not require the
  existence of a van Hove sequence in $G$. In fact, $\mu$ has
  $\cB$-sparse FBS if and only if $\mu$ is Fourier transformable,
  $\widehat{\mu}$ is a pure point measure, and the point set
  $\{ \chi \in \widehat{G} : \widehat{\mu}(\{ \chi \}) \neq 0 \}$ is
  $\cB$-sparse.

  Moreover, if a measure $\mu \in \cM^\infty(G)$ has sparse FBS,
  $\widehat{\mu}$ is pure point and, consequently, $\mu$ must be
  strongly almost periodic \cite[Cor.~4.10.13]{MoSt}.  Here, strongly
  almost periodic for a measure $\mu$ means that $\mu*g$ is uniformly
  (or Bohr) almost periodic for every $g\in C_{\mathsf{c}} (G)$,
  {and any such measure $\mu$ must be translation bounded.}
  In fact, {for $\mu\ne 0$,} $\supp(\mu) \subseteq G$ is
  relatively dense by \cite[Lemma~5.9.1]{NS11}.  In particular, it
  then follows that a measure $\mu$ with sparse FBS has Meyer set
  support if and only if $\mu \neq 0$ and $\supp(\mu)$ is a subset of
  a Meyer set. \exend
\end{remark}

\begin{example}\label{ex:cryst}
  All crystallographic measures on $\RR^d$ have a strongly sparse
  FBS. Indeed, $\omega \in \cM^{\infty} (\RR^d)$ is
  \emph{crystallographic} if it is of the form
  $\omega = \mu * \delta^{}_{\nts \vG}$ with $\mu$ a \emph{finite}
  measure and $\vG\subset \RR^d$ a lattice; compare
  \cite[Sec.~9.2.3]{TAO1}. Any such measure is Fourier transformable,
  with
\[
   \widehat{\omega} \, = \, \widehat{\mu} \cdot
   \widehat{\delta^{}_{\nts \vG}} \, = \, \dens (\vG)
   \, \widehat{\mu} \, \delta^{}_{\nts \vG^0} \, = \,
   \dens (\vG) \sum_{k\in \vG^0} \widehat{\mu} (k)
   \, \delta^{}_{k}
\]
by an application of the convolution theorem in conjunction with the
PSF from Eq.~\eqref{eq:PSF}. Here, $\widehat{\mu}$ is a continuous
function on $\RR^d$, and the dual lattice, $\vG^0$, is a uniformly
discrete point set; see \cite[Ex.~9.2]{TAO1}. This means that
$\supp (\widehat{\omega})\subseteq \vG^0$ is a strongly sparse point
set in $\RR^d$, and $\omega$ is \emph{doubly sparse} when $\mu$ has
finite support, which is to say that it is of the form
$\mu = \sum_{x\in F} \mu(\{ x \}) \, \delta_x$ for some finite set
$F \subset \RR^d$.  \exend
\end{example}

In what follows, we shall consider a slight generalisation of this
idea, namely, measures that are supported within finitely many
translates of a lattice, but with coefficients that are not
necessarily \mbox{lattice{\ts}-periodic}.  Such measures thus have a
support with a crystallographic structure, without actually being
crystallographic in the above sense.\smallskip

To continue, we need the following notion for continuous functions.
\begin{definition} 
  We say that a continuous function
  $h \! : \, H \xrightarrow{\quad} \CC$ has \emph{finite-measure
    support} if
  $\, \theta^{}_{\nts H} \bigl( \{ x \in H : h(x) \neq 0 \} \bigr) <
  \infty$.
\end{definition}

\begin{fact}\label{fact:h-is-L1}
  Let\/ $H$ be an arbitrary LCAG, and consider\/ $h \in \Cz(H)$. If\/
  $h$ has \mbox{finite{\ts}-measure} support, then\/ $h \in L^1 (H)$.
\end{fact}

\begin{proof}
  Any $h \in \Cz (H)$ is bounded. With
  $U\defeq \{ x \in H : h(x) \neq 0 \}$, we thus get
\[
   \int_{\nts H} \lvert h(z) \rvert \dd z  \, = 
   \int_U \lvert h(z) \rvert \dd z \, \leqslant \,
    \| h \|^{}_\infty  \, \theta^{}_H( U ) \, < \, \infty \ts ,
\]
which implies the claim.
\end{proof}

\section{Measures with Meyer set support and 
sparse FBS}\label{sec:meyer}

{In this section,} we characterise translation-bounded
measures, so $\mu\in\cM^{\infty} (G)$, with the additional properties
that $\supp (\mu) - \supp (\mu)$ is uniformly discrete and that $\mu$
has a sparse FBS.  Simple examples are Dirac combs of lattices in
$\RR^d$, as mentioned in Eq.~\eqref{eq:PSF} and in
Example~\ref{ex:cryst}.  An important tool will be the structure of
compactly generated LCAGs, which we recall for convenience from
\cite[Thm.~9.8]{HR}; see also \cite[Thm.~4.2.29]{Reiter}.

\begin{fact}\label{fact:comp-gen}
  If the LCAG\/ $H$ is compactly generated, there are non-negative
  integers\/ $d$ and\/ $m$ such that\/ $H$, as a topological group, is
  isomorphic with\/ $ \RR^d \!\times\! \ZZ^m \!\times\! \KK$, where
  the Abelian group\/ $\KK$ is compact.  \qed
\end{fact}

Recall that the $m$-torus, $\TT^m = \RR^m/\ZZ^m$, is the dual group of
$\ZZ^m$.  For convenience in explicit calculations, we represent it as
$[0,1)^m$ with addition modulo $1$, which is fully compatible with
writing the elements of $\TT^m$ as $x+\ZZ^m$ with $x\in [0,1)^m$.
Before we continue, we need the following simple variant of the
classic Paley--Wiener theorem \cite[Sec.~VI.4]{Y}.

\begin{lemma}\label{lem:PW}
  Let\/ $d,m\in\NN$ be fixed. Then, for any fixed\/
  $f \in C_{\mathsf{c}} (\RR^d \! \times \! \ZZ^m)$, there exists an
  analytic function\/
  $F \! : \, \RR^d \! \times \! \RR^m \xrightarrow{\quad} \CC$ such
  that\/ $F$ is\/ $\ZZ^m$-periodic in the second argument and that\/
  $\widehat{f} (x, y + \ZZ^m ) = F (x,y)$ holds for all\/ $x\in\RR^d$
  and\/ $y\in [0,1)^m$.
\end{lemma}

\begin{proof}
  Define
  $\mu \! : \, C_{\mathsf{c}} (\RR^d \! \times \! \RR^m )
  \xrightarrow{\quad} \CC$ by
\[
     \mu (g) \, = \sum_{v\in\ZZ^m} \int_{\RR^d} g(u,v)  f (u,v) \dd u \ts ,
\]   
which is a finite measure of compact support because $\supp (f)$ is
compact by assumption, and thus also a tempered distribution.  Its
(distributional) Fourier transform is an analytic function on
$\RR^{d+m}$, by the easy direction of the Paley--Wiener--Schwartz
theorem for distributions, see \cite[Thms.~III.2.2 and III.4.5]{GS},
and reads
\[
    F (x,y) \, = \,  \widehat {\mu} (x,y) \, = \sum_{v\in\ZZ^m}
    \ee^{-2 \pi \ii vy} \int_{\RR^d} 
    \ee^{- 2 \pi \ii ux} f (u,v) \dd u \ts ,
\]
with a unique continuation to an entire function on $\CC^{d+m}$.  The
relation $F (x, y+k) = F (x,y)$ for arbitrary $k\in \ZZ^m$ and all
$x\in \RR^d$, $y \in \RR^m$ is clear.
    
On the other hand, when $x\in\RR^d$ and $y+\ZZ^m \in \TT^m$, we get
\[
\begin{split}
    \widehat{f} (x, y+\ZZ^m) \, & = \int_{\RR^d}
    \sum_{v\in\ZZ^m} \ee^{-2 \pi \ii vy} \,
    \ee^{- 2 \pi \ii ux} f (u,v) \dd u \\[2mm] 
    & = \int_{\RR^d \times \ZZ^m} \ee^{-2 \pi \ii 
    (ux + vy) } \dd \mu (u,v) 
    \, = \, F (x,y) \ts ,
\end{split}    
\]    
   which completes the argument.
\end{proof}  

Here, we are interested in the following consequence.

\begin{coro}\label{coro:PW}
  Let\/ $d,m\in\NN$ be fixed and consider a function\/
  $f \in C_{\mathsf{c}} (\RR^d \! \times \ZZ^m)$ with\/
  $f \not \equiv 0$. Then, the set
\[
      U \, \defeq \, \{ ( x, y + \ZZ^m) \in \RR^d \! \times \! \TT^m :
      \widehat{f} (x, y + \ZZ^m) = 0 \}
\]   
   has measure\/ $0$ in\/ $\RR^d \! \times \! \TT^m$.
\end{coro}

\begin{proof}
  Defining the function $F$ as in Lemma~\ref{lem:PW},
  $V \defeq \{ (x,y) \in \RR^d \! \times \! \RR^m : F (x,y) = 0 \}$ is
  a null set in $\RR^d \! \times \! \RR^m$ because $F$ is analytic
  (see also \cite{mityagin} for this point). Since $F$ is
  $\ZZ^m$-periodic in its second variable, we have a canonical
  projection
  $\pi \! : \, \RR^d \! \times \! \RR^m \xrightarrow{\quad} \RR^d \!
  \times \! \TT^m$ such that $U = \pi (V)$ and $V = \pi^{-1} (U)$,
  which implies the claim.
\end{proof}

At this point, we can harvest the constructive approach to the CPS
of a given Meyer set with methods from \cite{BM,NS11}.

\begin{remark}\label{rem:harvest}
  Recall that a subset $\vL$ of a locally compact Abelian group $G$ is
  a Meyer set if it is relatively dense and
  $\vL\nts\nts -\! \vL \nts\nts - \!\vL$ is uniformly discrete \cite{Jeff2}.
  If $G$ is compactly generated, then the second condition is equivalent
  to the uniform discreteness of $\vL \nts \nts - \! \vL$; see
  \cite[Thm.~1.1]{Jeff} for $G=\RR^d$ and \cite[Appendix]{BLM} for the
  general case, as well as \cite[Rem.~2.1]{TAO1}. We note further that
  the following results do not require the stronger Meyer set
  condition; that is, in this case, it is sufficient to require that
  $\supp(\mu)\subseteq \vL$, where $\vL- \nts \vL$ is uniformly
  discrete; compare \mbox{\cite[Thm.~5.5.2]{NS11}}.  \exend
\end{remark}

\begin{prop}\label{P1} 
  Let\/ $\mu\neq 0$ be a translation-bounded measure on $G$.  If\/
  $\supp(\mu)$ is a subset of a Meyer set and if\/ $\mu$ has\/
  $\cB$-sparse FBS for some van Hove sequence\/ $\cB$ in\/
  $\widehat{G}$, then\/ $\mu$ is\/ $(\cA, \cB\ts)$-sparse for every
  van Hove sequence\/ $\cA$ in\/ $G$.
  
  Moreover, there is a CPS\/ $(G, \ZZ^m \! \times \!\KK, \cL)$, with\/
  $m\in\NN_0$ and\/ $\KK$ a compact Abelian group, and some\/
  $h \in \Cc ( \ZZ^m \! \times \! \KK )$ with\/
  $\widehat{h} \in C^{}_0 \bigl(\TT^m \! \times \! \widehat{\KK}
  \bigr) \cap L^1 \bigl(\TT^m \! \times \! \widehat{\KK} \bigr)$ such
  that
\[
   \mu \, = \, \omega^{}_h   \quad \text{and} \quad 
   \widehat{\mu} \, = \, \dens (\cL) \, \omega^{}_{\widecheck{h}}\ts .
\]
\end{prop}

\begin{proof}
  Since $\supp (\mu)$ is contained in a Meyer set, it is uniformly
  discrete, hence strongly sparse, and thus $\cA$-sparse for every van
  Hove sequence $\cA$ in $G$.

  By definition, compare Remark~\ref{rem:sparse},
  $\mu\in\cM^{\infty} (G)$ having $\cB$-sparse FBS means that $\mu$ is
  Fourier transformable and $\widehat{\mu}$ is a pure point measure.
  Consequently, by \cite[Cor.~4.10.13]{MoSt}, $\mu$ is a strongly
  almost periodic measure.  By Remark~\ref{rem:sparse}, we know that
  $\supp (\mu)$ must actually be a Meyer set, so
  \cite[Thm.~5.5.2]{NS11} implies that there exists a CPS
  $(G, H, \cL)$, with $H$ compactly generated, and some function
  $h \in \Cc(H)$ such that $\mu \, = \, \omega^{}_h$. In other words,
  $\mu$ is a weighted model comb.
  
  Since $\mu$ is Fourier transformable, \cite[Thm.~5.3]{CRS} implies
  that we have $\widecheck{h}\in L^1(\widehat{H})$ and
  $\widehat{\omega^{}_h} = \dens(\cL) \ts \omega_{\widecheck{h}}$,
  where $\widecheck{h} \in C^{}_0 (\widehat{H})$ is clear from
  \cite[Thm.~1.2.4]{Rudin} (or from the Riemann--Lebesgue lemma).  By
  assumption, $\supp(\omega^{}_{\widecheck{h}})=\supp(\widehat{\mu})$
  is $\cB$-sparse in $\widehat{G}$. Via the dual CPS
  $(\widehat{G}, \widehat{H}, \cL^0 )$, and applying
  Corollary~\ref{coro:cor1} to the set
  $U = \{ z \in \widehat{H} : \widecheck{h} (z) \ne 0 \}$, we see that
 \begin{equation}\label{eq:finiteU}
  \theta^{}_{\nts\widehat{H}}(U)<\infty \ts ,
 \end{equation}
 and that the FBS of $\mu$ is the cut and project set $\oplam (U)$ in
 the dual CPS.

 Now, by Fact~\ref{fact:comp-gen}, we have
 $H \cong \RR^d \!\times\! \ZZ^m \!\times\! \KK$ for some
 $d,m \in \NN_0$ and $\KK$ a compact Abelian group, and 
 we identify $H$
 with this group. We shall now show that, in fact, $d=0$.  Since
 $\mu\ne 0$ by assumption, we have $h\not\equiv 0$ and thus
 $\widecheck{h} (x^{}_{0}, y^{}_{0}, z^{}_{0}) \ne 0$ for some
 $ (x^{}_{0}, y^{}_{0}, z^{}_{0}) \in \RR^d \! \times \! \TT^m \!
 \times \! \widehat{\KK}$.  From \eqref{eq:finiteU}, we get
\begin{equation}\label{eq:finiteUsection}
     \theta^{}_{\nts\widehat{H}}
     \bigl(U\cap (\RR^d \! \times \! \TT^m
     \! \times \! \{ z^{}_{0} \})\bigr) \, < \, \infty \ts .
\end{equation}
Now, for each $x\in\RR^d,y\in\TT^m$, we have
\begin{equation}\label{eq:Ff=Fh0}
    \widecheck{h} (x,y, z^{}_{0} ) \, =
    \int_{\RR^d} \int_{\ZZ^m} \int_{\KK}
    \chi^{}_x(s) \, \chi^{}_y(t) \, 
    {\chi^{}_{z^{}_0}(u)} \, h(s,t,u)
    \dd\theta^{}_{\KK}(u) \dd\theta^{}_{\ZZ^m}(t)
    \dd\theta^{}_{\RR^d}(s) \, = \,
    \widecheck{f}(x,y) \ts ,
\end{equation}
where $f \! : \, \RR^d\times\ZZ^m \xrightarrow{\quad} \CC$
is defined by
\[
   (s,t) \, \mapsto \, f(s,t) \, \defeq
   \int_{\KK} {\chi^{}_{z^{}_0}(u)} 
    \, h(s,t,u) \dd\theta^{}_{\KK}(u) \ts ,
\]
which satisfies
$\widecheck{f} (x^{}_{0}, y^{}_{0} ) = \widecheck{h} (x^{}_{0},
y^{}_{0}, z^{}_{0} ) \ne 0$.  By Lemma~\ref{lem:PW}, the function
$\widecheck{f}$ is analytic and satisfies
$\widecheck{f} \not \equiv 0$.  Thus, by Corollary~\ref{coro:PW}, the
set
 \[
     Z \, \defeq \, \{(x,y)\in\RR^d \! \times\! 
         \TT^m : \widecheck{f}(x,y)=0\}
 \]
 has measure $0$ in $\RR^d \! \times \! \TT^m$. Let 
 \[
       V \, = \, \{(x,y)\in \RR^d \! \times \! \TT^m : 
         \widecheck{h}(x,y, z^{}_{0})\neq 0\} \ts ,
 \] 
 so that, by \eqref{eq:Ff=Fh0}, we have
 $\RR^d \nts\nts \times \nts\nts \TT^m = Z\ts\ts\dot{\cup}\ts\ts V$.
 Since $\widehat{\KK}$ is discrete, $\theta_{\widehat{\KK}}$ is
 proportional to counting measure, which means
 \[
       \theta^{}_{\RR^d \times\TT^m}(V) \, = \,
       c  \, \theta^{}_{\RR^d\times\TT^m\times\widehat{\KK}}
       (V\! \times \! \{ z^{}_{0} \})
 \]
 for some $c>0$, where the right-hand side is finite as a consequence
 of Eq.~\eqref{eq:finiteUsection}. We thus get
\[
     \theta^{}_{\RR^d \times \TT^m}
       (\RR^d \! \times \! \TT^m) \, = \,
      \theta^{}_{\RR^d \times \TT^m}
      (V\ts\ts\dot{\cup}\ts\ts Z) \, = \,  
      \theta^{}_{\RR^d \times \TT^m}
     (V) \,  < \, \infty \ts ,
\]
which is only possible if $d=0$. Consequently,
$H = \ZZ^m \! \times \! \KK$ together with
$\widehat{H} = \TT^m \!  \times \! \widehat{\KK}$, and our claims
follow.
\end{proof}

\begin{remark}\label{rem QUP}
  The last part of the above proof may alternatively be shown by
  invoking the \emph{qualitative uncertainty principle} (QUP) for
  LCAGs, as nicely summarised in \cite{HOG}.  Let $\KK_0$ be the
  (connected) identity component of $\KK$. Then,
  $\RR^d \!\times\! \{0\} \!\times\! \KK_0$ is the identity component
  of $H$.  Since $f$ and $\widehat{f}$ have \mbox{finite{\ts}-measure}
  support by assumption, with $f \not\equiv 0$, the QUP fails in
  $H$. By \cite[Thm.~1]{HOG}, the identity component of $H$ must then
  be compact, thus $d=0$.  \exend
\end{remark}

Recall that the Eberlein (or volume-averaged) convolution of two
measures $\mu, \nu \in \cM^{\infty} (G)$, relative to a given van
Hove sequence $\cA$, is defined by
\[
     \mu \circledast \nu \, = \lim_{n\to\infty}
     \frac{\mu^{}_{n} \nts * \ts \nu^{}_{n}}{\vol (A_n)} \ts ,
\]
where $\mu^{}_{n}$ and $\nu^{}_{n}$ are the restrictions of $\mu$
and $\nu$ to the set $A_n$. Here, the existence of the vague limit 
is assumed, which is
always the case in our setting. An explicit proof of the following
result is given in \cite[Prop.~5.1]{CRS}, and need not be repeated
here; see also \cite[Sec.~9.4]{TAO1} and \cite{BM}.

\begin{coro}\label{coro:L2} 
  Under the conditions of Proposition~\textnormal{\ref{P1}}, which
  comprise the transformability of\/ $\mu$, the autocorrelation\/
  $\gamma \defeq \mu \circledast \widetilde{\mu} $ is well defined,
  and one has the relations
\[  
   \gamma \, = \, \dens(\cL) \, \omega^{}_{h*\widetilde{h}}
   \quad \text{and} \quad
   \widehat{\gamma} \, = \, \dens (\cL)^2 \ts\ts 
   \omega^{}_{\lvert \widecheck{h} \vert^2} \ts .
\]    
Moreover, setting\/ $S\defeq \supp( \widehat{\mu})$, we also have the
representation\/
$ \widehat{\mu} = \sum_{y \in S} \ts\widehat{\mu} ( \{y \}) \,
\delta_y$ together with\/
$ \widehat{\gamma} = \sum_{y \in S} \ts\lvert \widehat{\mu}( \{y \})
\rvert^2 \ts \delta_y $.  \qed
\end{coro}

We are now ready to formulate our first main result.

\begin{theorem}\label{thm:conj1}  
  Let\/ $ \mu \in \cM^\infty(G)$ with\/ $\mu \ne 0$ be such that\/
  $\supp(\mu)$ is contained in a Meyer set and that the FBS of\/ $\mu$
  is\/ $\cB$-sparse for some van Hove sequence\/ $\cB$ in\/
  $\widehat{G}$. Then, there is a lattice\/ $\vG$ in\/ $G$ together
  with finite sets\/ $F \subseteq G$ and\/ $F' \subseteq \widehat{G}$
  such that
\[
   \supp(\mu) \, \subseteq \, \vG+F 
   \quad \text{and} \quad
   \supp(\widehat{\mu}) \, \subseteq \, \vG^{\ts 0}+F'.
\]
\end{theorem}

\begin{proof}
  By assumption, we have $\mu \ne 0$.  By Proposition~\ref{P1}, there
  exists a CPS $(G, H, \cL)$, with $H\defeq \ZZ^m \! \times \!\KK$,
  and some $h \in \Cc( \ZZ^m \!\times\! \KK)$ with
  $\widehat{h} \in L^1(\TT^m \!\times\! \widehat{\KK})$ continuous
  such that $\mu = \omega^{}_h $. Now, consider
\[
   \vG \, \defeq \ts  \oplam 
   \bigl( \{ 0 \} \!\times\! \KK \bigr)  .
\]
Since $H^{}_0 \defeq \{ 0 \} \!\times\! \KK $ is a subgroup of $H$,
and the $\star$-map is a group homomorphism, $\vG$ is a subgroup of
$G$. Moreover, since $H^{}_0$ is both compact and open, $\vG$ is a
Delone set. This shows that $\vG$ is a lattice in $G$.

Next, since $\supp(h)$ is compact, it can be covered by finitely many
translates of the open set $H^{}_0$. More precisely, there is a finite
set $S \subset \ZZ^m$ such that
$\supp (h) \subseteq \bigcup_{t\in S} \bigl( (t,0) + H^{}_{0}\bigr) $.
If we set $F \defeq \oplam(S \! \times \! \{0\})$, we see that $F$ is
finite and
\[
   \supp(\mu) \, = \, \oplam \bigl( \supp (h) \bigr)
   \, \subseteq \, \oplam (H^{}_{0}) +
   \oplam ( S \! \times \! \{ 0 \} )
   \, = \,  \vG + F \ts .
\]

To gain the corresponding result for $\widehat{\mu}$, we need to show
that $\supp \bigl(\widecheck{h} \bigr)$ is compact.  From the above,
we know that $\supp(h) \subseteq S \!\times\! \KK$. For $t\in S$, set
\[
   h^{}_t (\xi) \, \defeq \, h(t , \xi ) \ts ,
\]
so that $h^{}_t \in C(\KK)$ and
$h = \sum_{t \in S} 1^{}_{\{ t \}} \nts \otimes h^{}_t $. For any
$t\in S$ and $y\in\widehat{\KK}$, we have
\[
    \widecheck{h^{}_t }(y) \, =
    \int_{\KK} \chi^{}_y (u) \, h^{}_t (u) \dd\theta^{}_{\KK} (u) 
    \,\in \, \CC \ts .
\]    
Then, for arbitrary
$(x,y)\in\widehat{H}=\TT^m \!  \times\!\widehat{\KK}$, a simple
calculation shows that
\begin{equation}\label{eq:Fh}
    \widecheck{h}(x,y) \, = \sum_{t \in S}
    \chi^{}_x (t) \, \widecheck{h^{}_{t}} (y)  \ts .
 \end{equation}

 Fix $y\in\widehat{\KK}$ and define
 $g^{}_y \! : \, \TT^m \xrightarrow{\quad} \CC$ by
 $x \mapsto g^{}_y (x) = \widecheck{h}(x,y)$.  Next, for each
 $t \in S$, define $\chi^{}_{t} \! : \, \TT^m \xrightarrow{\quad} \CC$
 by $\chi^{}_{t} (x) = \chi^{}_{x} (t)$.  Note that $\chi^{}_{t}$
 simply is $t\in S\subset \ZZ^m$ viewed as a character on
 $\TT^m = \widehat{\ZZ^m}$.  Then, by Eq.~\eqref{eq:Fh}, we have
\[
    g^{}_y (x) \, = \sum_{t \in S}
    \widecheck{h^{}_{t}} (y) \, \chi^{}_{t} (x)
\]
for each $x\in\TT^m$, so that $g^{}_{y}$, for any fixed $y$, is a
trigonometric polynomial on $\TT^m$.  Applying Lemma~\ref{lem:PW} in
conjunction with Corollary~\ref{coro:PW}, we see that either
$g^{}_y \equiv 0$, or the set of zeros of $g_y$ is a null set in
$\TT^m$.  Now, consider
\[
     U_y \, \defeq \, \{ x\in\TT^m : \widecheck{h}(x,y)\neq 0\} 
     \, = \, \{x\in\TT^m : g^{}_y (x) \neq 0\} \ts .
\]
By the above, we see that either $U_y = \varnothing$ or
$\theta^{}_{\TT^m}(U_y) = 1$.
 
Since $\widehat{\KK}$ is discrete, we may repeat this process to
obtain such a set $U_y$ for each $y\in\widehat{\KK}$.  Then, for every
$y\in\widehat{\KK}$, we have either
$U_y \! \times \! \{y\} = \varnothing$ or
$\theta^{}_{\TT^m \times \KK}(U_y \!\times\! \{y\}) = 1$.  Next,
consider
\[
     J \, \defeq \, \{y\in\widehat{\KK} : U_y \neq \varnothing\} \ts .
\]
Recall from Eq.~\eqref{eq:finiteU} that the set
$U=\{z\in\TT^m \!\times \! \widehat{\KK} : \widecheck{h}(z)\neq 0\}$
has finite measure. We have
\[
     U \, = \bigcup_{y\in\widehat{\KK}} 
     U_y \!\times\! \{y\} \, = \bigcup_{y\in J} 
     U_y \! \times \! \{y\}
\]
and thus, since $\theta^{}_{\TT^m}(U_y) = 1$ for all $y\in J$ but
$\theta^{}_{\nts \widehat{H}} (U) < \infty$, we conclude that $J$ is a
finite set.  Noting that $U_y = \varnothing$ for $y\notin J$, we see
that, for any $y\notin J$, we have
\[
     \widecheck{h} (x,y) = 0 \quad \text{for all } x\in \TT^m .
\]
This implies $\supp\bigl(\widecheck{h}\bigr)\subseteq \TT^m\times J$
and, reasoning as we previously did for $\mu$, we find that
\[
    \supp(\widehat{\mu}) \, \subseteq \,
    \vG^{\ts\prime} + F^{\ts\prime} ,
\]
where $\vG^{\ts\prime}=\oplam(\TT^m\nts\nts\times\nts\nts\{0\})$ is a
lattice in $\widehat{G}$ and $F^{\prime}=\oplam(\{0\}\!\times\! J)$ is
a finite set, this time referring to the dual CPS,
$(\widehat{G}, \widehat{H}, \cL^{0})$.

To finish the proof, we need to show that
$\vG^{\ts\prime} = \vG^{\ts 0}$, where the lattice $\vG^{\ts 0}$ is
the annihilator of $\vG$. Recall that we had
$\vG= \oplam \bigl( \{ 0 \} \! \times \! \KK\bigr)$ and
$\vG^{\ts\prime}= \oplam \bigl( \TT^m \! \times \!  \{ 0 \}
\bigr)$. Let $y\in \vG^{\ts\prime}$, which means
$(y,y^{\star})\in\cL^0$ together with
$y^{\star} \in \TT^m \! \times \! \{0\}$. Similarly, $x\in \vG$ means
$(x,x^{\star})\in\cL$ with $x^{\star} \in \{ 0 \} \! \times \!
\KK$. But $(x,x^{\star}) \in \cL$ and $(y,y^{\star})\in\cL^0$ implies
$ \chi^{}_{y} (x) \, \chi^{}_{y^{\star}} (x^{\star}) = 1 $.

Now, $x^{\star} \in \{ 0 \} \! \times \! \KK$ gives us
$x^{\star} = (0,\xi)$ with $\xi \in \KK$, while
$y^{\star} \in \TT^m \! \times \! \{ 0 \}$ implies the form
$y^{\star} = (\eta, 0)$ with $\eta \in \TT^m$. Then,
$\chi^{}_{y^{\star}} (x^{\star}) = \chi^{}_{\eta} (0) \, \chi^{}_{0}
(\xi) = 1$. Employing the previous relation, we are thus left with
$\chi^{}_{y} (x) = 1$, which implies $y\in \vG^{\ts 0}$ because
$x\in \vG$ was arbitrary. Since this works for any
$y\in \vG^{\ts\prime}$, we have
$\vG^{\ts\prime} \subseteq \vG^{\ts 0}$.

To establish the converse inclusion, let $k\in \vG^{\ts 0}$ be
arbitrary but fixed, so $\chi^{}_{k} (x) =1$ for all
$x\in \vG = \oplam \bigl( \{0\}\!\times\! \KK\bigr)$. We work with the
CPS $(G,H,\cL)$ from above, and write elements of
$H=\ZZ^m \! \times \! \KK$ as $(t , \kappa)$.  Since
$\pi^{}_{H} (\cL)$ is dense in $H$ and $\{ t \} \! \times \!  \KK$ is
open in $H$ for any $t\in\ZZ^m$, we may conclude
\[
      \pi^{}_{\ZZ^m} (\cL) \, = \, \ZZ^m ,
\]
which implies that, for any $t\in\ZZ^m$, there are elements $x\in G$
and $\kappa\in\KK$ such that $(x,t,\kappa) \in \cL$.

Define the mapping $\psi \! : \, \ZZ^m \xrightarrow{\quad} \SSS^1$ by
$\psi (t) = \chi^{}_{k} (x)$, which turns out to be well
defined. Indeed, if $(x^{}_{1}, t, \kappa^{}_{1})$ and
$(x^{}_{2}, t, \kappa^{}_{2})$ are both elements of $\cL$, then so is
their difference, where we have
$(x^{}_{1} - x^{}_{2})^{\star} = (0, \kappa^{}_{1} - \kappa^{}_{2})
\in \{ 0 \} \! \times \! \KK$.  But this implies
$x^{}_{1} - x^{}_{2} \in\oplam \bigl( \{0\}\!\times\!  \KK\bigr)=
\vG$, so $\chi^{}_{k} (x^{}_{1} - x^{}_{2})=1$ due to
$k\in \vG^{\ts 0}$, and hence
$\chi^{}_{k} (x^{}_{1}) = \chi^{}_{k} (x^{}_{2})$.

Next, we show that $\psi$ defines a character on $\ZZ^m$.  Since
$\ZZ^m$ carries the discrete topology, $\psi$ is continuous.  Now, for
any $t^{}_{1}, t^{}_{2} \in \ZZ^m$, there are
$x^{}_{1}, x^{}_{2} \in G$ and $\kappa^{}_{1}, \kappa^{}_{2} \in \KK$
such that $(x^{}_i, t^{}_i, \kappa^{}_i) \in \cL$, and we get
$\psi (t^{}_i) = \chi^{}_{k} (x^{}_i)$ by definition. Since the sum of
two lattice points is again a lattice point, we also get
$\psi (t^{}_{1} + t^{}_{2}) = \chi^{}_{k} (t^{}_{1} + t^{}_{2}) =
\chi^{}_{k} (t^{}_{1}) \, \chi^{}_{k} ( t^{}_{2}) = \psi^{}_{k}
(x^{}_{1}) \, \psi^{}_{k} (x^{}_{2})$ as required.
 
Finally, since $\psi$ is a character on $\ZZ^m$, there exists an
element $\ell\in\TT^m$ such that $\psi = \chi^{}_{\ell}$. We now claim
that $(k,-\ell,0) \in \cL^0$. Indeed, for all $(x,t,\kappa) \in \cL$,
we have
\[
     \chi^{}_{k} (x) \, \chi^{}_{-\ell} (t) \, \chi^{}_{0} (\kappa)
     \, = \, \chi^{}_{k} (x) \, \overline{\chi^{}_{\ell} (t)} \, = \,
     \chi^{}_{k} (x) \, \overline{\psi (t)} \, = \,
     \chi^{}_{k} (x) \, \overline{\chi^{}_{k} (x)} \, = \, 1 \ts .
\]
This also means that
$k \in \oplam \bigl( \TT^m \! \times \! \{ 0 \} \bigr) = \vG^{\ts\prime}$, 
which completes the argument.
\end{proof}

For any fixed $y \in J$, with the set $J$ from the proof of
Theorem~\ref{thm:conj1}, the function defined by
$x \mapsto \widecheck{h} (x,y)$ is a trigonometric polynomial on
$\TT^m$. In fact, we can say more.

\begin{lemma}
  Let \/$\vG\subseteq G$ be a lattice and\/ $\mu$ a
  Fourier-transformable measure on\/ $G$ such that\/
  $\supp(\mu)\subseteq \vG+F$ and\/
  $\supp (\widehat{\mu})\subseteq \vG^{\ts 0} + F^{\prime}$ for finite
  sets\/ $F\subseteq G$ and\/ $F^{\prime}\subseteq\widehat{G}$. Then,
  there is a set\/
  $ \{ \tau^{}_{1}, \ldots , \tau^{}_{N} \} \subseteq F$ such that one
  can represent\/ $\mu$ as
\[ 
     \mu \, = \sum_{j=1}^N\sum_{x\in \vG+\tau^{}_j} 
        P_j (x) \, \delta_x \ts ,
\]
   where each\/ $P_j$ is a trigonometric polynomial on\/ $G$. 
\end{lemma}

\begin{proof}
  Given a lattice $\vG$ and a finite set $F$, there exists a minimal
  finite set, $F_0\subseteq F$ say, such that $\vG+F_0=\vG+F$.
  Without loss of generality, we may assume that $F$ and $F^{\prime}$
  are minimal in this sense.  Then, applying \cite[Rem.~5]{Nicu2} to
  the measure $\gamma = \widehat{\mu}$, we gain the existence of a
  finite measure $\nu$ on $G$ such that
\[ 
    \mu^{\dagger} \, = \, \widehat{\widehat{\mu}\ts}
    \, = \, \biggl(\,\sum_{x\in \vG} 
    \sum_{\chi\in F^{\prime}} \chi(x) \,
    \delta_x \! \biggr) \nts \ast \nu \ts ,
\]
where $\mu^{\dagger} (g) \defeq \mu (g \circ I)$ with $I (x) = -x$, so
$(\mu^{\dagger})^{\dagger} = \mu$ and
$(\mu*\nu)^{\dagger} = \mu^{\dagger} * \nu^{\dagger}$.  Consequently,
with $(\delta_x)^{\dagger} = \delta_{-x}$ and
$\chi (-x) = \overline{\chi (x)}$, one gets
\[ 
   \mu \, = \biggl( \, \sum_{x\in \vG} \sum_{\chi\in F^{\prime}}
    \overline{\chi(x)}\, \delta_x \! \biggr) \nts \ast \nu^{\dagger}  .
\]

Define the measures
$ \nu^{}_1 \defeq \sum_{x\in \vG+F} \nu^{\dagger} \bigl( \{x\} \bigr)
\ts \delta_x$ and $\nu^{}_2\defeq \nu^{\dagger} \nts -
\nu^{}_1$. Then,
\[ 
   \mu \, =  \, \biggl( \, \sum_{x\in \vG} \sum_{\chi\in F^{\prime}}
   \overline{\chi(x)} \, \delta^{}_x \! \biggr)
   \nts \ast \nu^{}_1 \: + \,
   \biggl(\, \sum_{x\in \vG} \sum_{\chi\in F^{\prime}}
    \overline{\chi(x)} \, \delta^{}_x \! \biggr) \nts \ast \nu^{}_2 
    \; \eqdef \; \mu^{}_1 + \mu^{}_2 \ts .
\]
Since $\mu$ and $\mu^{}_1$ are supported in $\vG+F$, we have
$\supp (\mu^{}_2)\subseteq \vG+F$. Observe that $\nu^{}_2 (\{x\}) = 0$
for all $x\in \vG+F$ by construction, which implies
$\mu^{}_2 (\{x\}) = 0$ for all $x\in \vG+F$ by a simple
calculation. Consequently, $\mu^{}_2 = 0$ as a measure, and we have
\[ 
   \mu \, = \, \biggl( \, \sum_{x\in \vG} \sum_{\chi\in F^{\prime}}
    \overline{\chi(x)}\, \delta^{}_x \! \biggr) \nts \ast \nu^{}_1 \ts ,
\]
where $\nu^{}_1$ is a finite pure point measure with
$\supp (\nu^{}_{1} ) \subseteq \vG+F$.

Now, let $F = \{\tau^{}_1, \ldots , \tau^{}_N \}$.  Then, recalling
$\delta_x * \delta_y = \delta_{x+y}$, we can explicitly write
\[
   \nu^{}_1  \, = \sum_{j=1}^N \biggl( \, 
      \sum_{x\in \vG} \nu^{}_1 \bigl( \{x+\tau_j\} \bigr)
          \, \delta^{}_{x} \! \biggr) \nts \ast\delta^{}_{\tau_j}
     \, = \sum_{j=1}^N \eta^{}_j \ast\delta_{\tau_j} 
\]
with
$\eta^{}_j\defeq\sum_{x\in \vG} \nu^{}_1 \bigl( \{x+\tau_j\} \bigr)
\delta_{x}$.  Note that each $\eta^{}_j$ is supported inside $\vG$ and
is finite, meaning that
$\sum_{x\in \vG} \big\lvert \eta^{}_j \bigl( \{x\} \bigr) \big\rvert <
\infty $.  This finiteness allows us to choose (and change) the order
of summation in what follows.  Now, decompose $\mu$ as
$\mu = \sum_{j=1}^{N} \vartheta^{}_{j} * \delta^{}_{\tau_j}$ with
\[ 
   \vartheta^{}_{j} \, = \, \biggl( \, \sum_{x\in \vG}
   \sum_{\chi\in F^{\prime}} \overline{\chi(x)}\,
    \delta^{}_{x}\! \biggr)\nts \ast \eta^{}_{j} \ts .
\]
Then, with
$\eta^{}_{j} \bigl( \{ x \} \bigr) = \nu^{}_{1} \bigl( \{ x +
\tau^{}_{j} \} \bigr)$ for $x\in \vG$, we obtain
\[
 \begin{split}
     \vartheta^{}_{j} \, & = 
     \sum_{y\in \vG} \eta^{}_{j} \bigl( \{ y \} \bigr)
     \sum_{x\in \vG}  \sum_{\chi\in F^{\prime}} \overline{\chi(x)}
     \, \delta^{}_{x+y} \,
   = \sum_{z\in \vG} \sum_{\chi\in F^{\prime}} \sum_{y\in \vG} 
         \eta^{}_{j} \bigl( \{ y \} \bigr) \chi(y) \,
         \overline{\chi(z)} \, \delta^{}_{z} \ts ,
 \end{split}
\]
where we have used a change of variable transformation in the lattice
$\vG$ together with the relation $\vG-\vG=\vG$.  With
$a^{}_{\chi} \defeq \sum_{y\in \vG} \eta^{}_{j} \bigl( \{ y \} \bigr)
\chi(y)$ for $\chi \in F'$, we get
\[ 
   \vartheta^{}_{j}  \, = \sum_{z\in \vG} \sum_{\chi\in F^{\prime}} 
   a^{}_{\chi} \,  \overline{\chi}(z) \, \delta^{}_z  \, = \sum_{z\in \vG}
   P^{}_{j} (z + \tau^{}_{j} ) \, \delta^{}_{z} \ts ,
\]
where
$P^{}_{j} \defeq \sum_{\chi\in F^{\prime}} a^{}_{\chi} \, \chi
(\tau_j) \, \overline{\chi}$ is a finite linear combination of
characters from $\widehat{G}$, and thus a trigonometric polynomial on
$G$. Consequently, we also have
\[
    \vartheta^{}_{j} \ast \delta^{}_{\tau_j} \,
     = \! \sum_{z\in \vG+\tau_j} \! P^{}_{j} (z) \, \delta^{}_{z}  \ts .
\]
Repeating this construction for each $j$, we find
\[
   \mu \, = \sum_{j=1}^N \vartheta^{}_{j} \ast \delta^{}_{\tau^{}_j}
   \, = \sum_{j=1}^{N} \sum_{z\in \vG+\tau_j} \! P^{}_{j} (z) \, 
   \delta^{}_{z} \ts ,
\]
which was our original claim.
\end{proof}

Now, consider a measure $\mu\in\cM^{\infty} (G)$ with Meyer set
support that is Fourier transformable. By \cite[Thm.~5.9]{CRS} or
\cite[Thm.~4.9.32]{MoSt}, we know that $\widehat{\mu}$ is
transformable as well. Then, inter\-changing the roles of $G$ and
$\widehat{G}$ as well as those of $\mu$ and $\widehat{\mu}$, the
previous result can be derived for $\widehat{\mu}$, this time with
$\{ \sigma^{}_1, \ldots , \sigma^{}_{\nts M} \} \subseteq F'$. This
leads to the following general statement.

\begin{theorem}\label{thm:trig-coeffs}
  Under the conditions of Theorem~\textnormal{\ref{thm:conj1}}, we can
  write
\[
    \mu \, = \sum_{i=1}^N \sum_{x \in \vG+ \tau^{}_{i}} \!
     P^{}_{i} (x) \, \delta^{}_{x} \qquad \text{and} \qquad
    \widehat{\mu} \, = \sum_{j=1}^{M} 
    \sum_{y \in \vG^{\ts 0}+ \sigma^{}_j} \! Q^{}_{j} (y) \, \delta^{}_{y} \ts ,
\]
where each\/ $P^{}_{i}$, respectively\/ $Q^{}_{j}$, is a trigonometric
polynomial on the group\/ $G$, respectively\/ $\widehat{G}$, while\/
$N$ and\/ $M$ are the cardinalities of the minimally chosen finite
sets\/ $F$ and\/ $F'$.  \qed
\end{theorem}

\begin{remark}
  Note that the polynomials $P^{}_i$ and $Q^{}_j$ in
  Theorem~\ref{thm:trig-coeffs} are not unique. Indeed, if $\chi$ is
  any character that is constant on $\vG+\tau^{}_{i}$ (meaning
  $\chi \in \vG^{\ts 0}$) and if $c$ is the corresponding constant,
  then $(\overline{c} \ts\ts \chi ) P^{}_{j}$ is another polynomial
  that agrees with $P_j$ on $\vG+\tau^{}_{i}$. A (somewhat) canonical
  choice for the polynomials can now be made as follows.

  Given polynomials $P^{}_i$ such that the first relation in
  Theorem~\ref{thm:trig-coeffs} holds, there exist characters
  $\chi^{}_{1}, \ldots , \chi^{}_{M}$ and coefficients $c^{}_{ij}$
  with $1\leqslant i \leqslant N$ and $1\leqslant j \leqslant M$ such
  that $ P^{}_{i} = \sum_{j=1}^{M} c^{}_{ij} \ts\ts \chi^{}_{j}$.
  This is possible because the $\chi^{}_j$ can comprise \emph{all}
  characters which appear in the polynomials $P^{}_i$, then with the
  possibility that some of the coefficients $c^{}_{ij}$ vanish.

Now, using $\psi^{}_{\tau^{}_{i}} (\chi) \defeq \chi (\tau^{}_{i})$,
a simple computation gives
\[
\begin{split}  
  \widehat{\mu} \, & = \, \biggl( \sum_{i=1}^{N} \,
  \biggl( \sum_{j=1}^{M} c^{}_{ij} \, \chi^{}_{j} \biggr)
  \bigl( \delta^{}_{\vG} * \delta^{}_{\tau^{}_{i}} \bigr)
  \biggr)^{\!\widehat{\phantom{vv}}} 
   = \: \dens (\vG) \sum_{i=1}^{N}\, \biggl( \sum_{j=1}^{M}
  c^{}_{ij} \, \delta^{}_{\chi^{}_{j}} \biggr) * \bigl(
  \psi^{}_{\tau^{}_{i}} \, \delta^{}_{\vG^{0}}\bigr) \\[2mm]
  & = \, \dens (\vG) \sum_{i=1}^{N} \sum_{j=1}^{M}
  c^{}_{ij} \, \chi^{}_{j} (\tau^{}_{i})\,
  \psi^{}_{\tau^{}_{i}} \ts\ts \delta^{}_{\vG^{0}+\chi^{}_{j}}
  \, = \sum_{j=1}^{M} \biggl( \sum_{i=1}^{N} \dens (\vG) \,
  c^{}_{ij} \, \chi^{}_{j} (\tau^{}_{i}) \,
  \psi^{}_{\tau^{}_{i}} \biggr)
  \delta^{}_{\vG^{0} + \chi^{}_{j}} .
\end{split}  
\]
We thus see that there are coefficients $c^{}_{ij}$, translations
$\tau^{}_{1}, \ldots, \tau^{}_{N}$ and characters
$\chi^{}_{1}, \ldots , \chi^{}_{M}$ such that we get
\[
  \mu \, = \sum_{i=1}^{N} P^{}_{i} \, \delta^{}_{\vG + \tau^{}_{i} }
  \quad \text{and} \quad \widehat{\mu} \, =
  \sum_{j=1}^{M} Q^{}_{j} \, \delta^{}_{\vG^{0} + \chi^{}_{j} }
\]
for the choice
\[
  P^{}_{i} \, = \sum_{j=1}^{M} c^{}_{ij} \ts\ts \chi^{}_{j}
  \quad \text{and} \quad Q^{}_{j} \, = \, \dens (\vG)
  \sum_{i=1}^{N} c^{}_{ij} \, \chi^{}_{j} (\tau^{}_{i}) \,
  \psi^{}_{\tau^{}_{i}} \ts ,
\]
which seems to be a reasonable standardisation. \exend
\end{remark}

The results of this section may be applied more generally.  Firstly,
we consider a transformable measure $\mu$ whose Fourier transform may
have a continuous part. We refer to the pure point part of a measure
$\nu$ by $\nu_{\mathsf{pp}}$, and use $\mu_{\mathsf{s}}$ and
$\mu^{}_0$ for the strongly almost periodic and null-weakly almost
periodic parts of a weakly almost periodic measure $\mu$; see
\cite[Sec.~4.10]{TAO2} for definitions and properties around the
underlying Eberlein decomposition,
$\mu = \mu_{\mathsf{s}} + \mu^{}_{0}$.

\begin{coro}\label{coro:mixed-FT}
  Let\/ $\mu$ be a translation-bounded, Fourier-transformable measure
  on\/ $G$ such that\/ $\ts\supp(\mu)$ is a subset of a Meyer set and
  that\/ $\ts\supp(\widehat{\mu}_{\mathsf{pp}})\neq\varnothing$ is\/
  $\cB$-sparse for some van Hove sequence\/ $\cB$ in\/
  $\widehat{G}$. Then, there exists a lattice\/ $\vG$ together with
  finite sets\/ $F=\{\tau^{}_1, \ldots, \tau^{}_N\}$ and\/ $F^{}_{0} $
  in\/ $G$ such that
\[
    \mu_{\mathsf{s}} \, = 
    \sum_{i=1}^N \sum_{x \in \vG+ \tau^{}_{i}}
     P^{}_i (x) \, \delta^{}_x  \quad \textrm{and} \quad
    \supp(\mu^{}_0) \, \subseteq \, \vG+F^{}_{0} \ts .
\]
Moreover, one has\/
$\supp (\widehat{\mu}^{}_{\mathsf{pp}}) \subseteq \vG^{\ts 0} + F'$ for
some finite set\/ $F' \subseteq \widehat{G}$.
\end{coro}

\begin{proof}
  Since $\mu$ has Meyer set support, then so does $\mu_{\mathsf{s}}$,
  as follows from \cite{NS-weight,Nicu2}.  Now, the first part of the
  claim becomes a consequence of Theorem~\ref{thm:trig-coeffs}.  For
  the second part, we know that $\supp(\mu)$ can be embedded inside a
  model set $\vL\subseteq G$ with compact window. Then,
  $\supp(\mu^{}_0), \ts \supp(\mu_{\mathsf{s}}) \subseteq \vL$.  Since
  $\mu_{\mathsf{s}}$ is strongly almost periodic and non-trivial,
  $\supp(\mu_{\mathsf{s}})$ is relatively dense. So, applying
  \cite[Lemma~5.5.1]{NS11}, we find a finite set,
  $F^{}_1 \subseteq G$, such that
\[
    \vL \, \subseteq \, \supp(\mu_{\mathsf{s}})+ F^{}_1 \ts .
\]
From the first part, $\supp(\mu_{\mathsf{s}}) \subseteq \vG+F$, so we
have
\[
    \supp(\mu^{}_0) \, \subseteq \, \vL
    \, \subseteq \, \supp(\mu_{\mathsf{s}}) + F^{}_1 
    \, \subseteq \, (\vG+F)+ F^{}_1\ts .  
\]
The last claim follows from Theorem~\ref{thm:conj1}.  
\end{proof}

In general, a measure need not (and generally will not) be Fourier
transformable in order to possess an autocorrelation and a
diffraction. {We can extend these results to the class of
  weakly almost periodic measures; see \cite[Sec.~4.10]{TAO2} for
  definitions and \cite{MoSt,LS2} for details. This class contains all
  translation-bounded, Fourier-transformable measures \cite{MoSt}.}
Next, we require the concept of Fourier--Bohr coefficients; compare
\cite[Def.~2.18]{LS2}.

\begin{definition}\label{def:FB-coeffs}
  Let $\mu$ be a weakly almost periodic measure on a group $G$.  The
  \emph{Fourier--Bohr coefficients} of $\mu$ are defined, for each
  $\chi\in\widehat{G}$, by
\[
   c_{\chi}(\mu) \, \defeq \,
    \MM \bigl( \overline{\chi} \ts \mu \bigr)
   \, = \lim_{n\to\infty}\frac{\bigl( \overline{\chi}\ts\mu 
       \bigr) (A_n)}{\vol (A_n)} \ts ,
\]
where $\cA = \{A_n\}$ is any van Hove sequence in $G$. 
\end{definition}

Note that the existence of the Fourier--Bohr coefficients, in the
above form, follows from \mbox{\cite[Lemma~4.10.7]{MoSt}}, and their
values do not depend on the van Hove sequence chosen; see also
\cite{ARMA}.

\begin{coro}\label{coro:wap}
  Let\/ $\mu$ be a weakly almost periodic measure on\/ $G$ such that\/
  $\supp(\mu)$ is contained in a Meyer set.  If\/
  $\mu_{\mathsf{s}} \neq 0$ and the set\/
  $S\defeq\{ \chi\in\widehat{G} : c_{\chi}(\mu)\neq 0 \}$ is\/
  $\cB$-sparse in\/ $\widehat{G}$ for some van Hove sequence\/ $\cB$
  in\/ $\widehat{G}$, there exists a lattice\/ $\vG$ in $G$ together
  with a finite set\/ $F\subseteq G$ such that
\[ 
    \supp(\mu_{\mathsf{s}}) \ts , \: 
    \supp(\mu^{}_0)  \, \subseteq \, \vG+F \ts ,
\]
together with\/ $S \subseteq \vG^{\ts 0} + F'$ for some
finite set\/ $F' \subseteq \widehat{G}$.
\end{coro}

\begin{proof}
  By \cite[Thm.~7.6]{LS2}, $\mu$ has the unique autocorrelation
  $\gamma^{}_{\mu}$ and pure point diffraction with support $S$. Since
  $\gamma^{}_{\mu}$ is translation bounded, transformable and
  supported in the Meyer set $\supp(\mu)-\supp(\mu)$, we may apply
  Corollary \ref{coro:mixed-FT} to $\gamma^{}_{\mu}$ to obtain
\[
   \supp (\gamma^{}_{\mu}) \, \subseteq \, \vG + F \ts .
\]

Next, let $\vL$ be any Meyer set in $G$ such that
$\supp (\mu) \subseteq \vL$. Then,
$\supp (\gamma^{}_{\mu}) \subseteq \vL - \vL$, and
\cite[Lemma~5.5.1]{NS11} guarantees the existence of a finite set
$F_1$ such that
\[
   \vL - \vL \, \subseteq \, \supp (\gamma^{}_{\mu} ) + F^{}_1 \ts .
\]
Fix some $s\in \vL$. Then, we get
\[
  \supp (\mu) \, \subseteq \, \vL \, \subseteq \,
  (\vL - \vL) + s \, \subseteq \, \supp (\gamma^{}_{\mu})
  + F^{}_1 + s \, \subseteq \, \vG + F + F^{}_1 + s \ts .
\]
The claim now follows from Corollary~\ref{coro:mixed-FT}.
\end{proof}

\section{Sparseness of positive definite
  measures}\label{sec:sap}

This is the moment where we need to recall further concepts of almost
periodicity and study their consequences and relations in the context
of measures with sparse supports. Since we are interested in
{measures with pure point Fourier transform}, 
it is natural to begin with the
class of strongly almost periodic measures, or $\ts\SAP$-measures for
short, which were defined in Remark~\ref{rem:sparse} and have been
used several times already.

\subsection{Positive definite  measures with uniformly discrete
  support}

Here, we consider positive definite, pure point $\ts\SAP$-measures,
where we are able to tighten the type of almost periodicity when the
support is uniformly discrete.  In particular, we will show that any
such measure is also sup-almost periodic.

For any pure point measure $\omega$ on $G$, one can consider
\[
    \| \omega \|^{}_{\infty} \, \defeq \, \sup_{x\in G}\, 
    \lvert \omega \rvert \bigl( \{ x \} \bigr) ,
\]    
which defines a norm on $\cM^{\infty}_{\mathsf{pp}}$, the space of
translation-bounded pure point measures.  Now, recall from
\cite[Def.~5.3.4]{NS11} that a measure $\mu$ from this class is called
\emph{sup-almost periodic} if the set
$P_{\varepsilon} \defeq \{ t\in G: \| \mu - T^{}_{t} \ts \mu
\|^{}_{\infty} < \varepsilon \}$ is relatively dense for every
$\varepsilon > 0$.

\begin{theorem}\label{thm:ap}
  Let\/ $0\ne \mu =\sum_{x \in \vL} a (x) \, \delta_x$ be a positive
  definite, strongly almost periodic
  {$($hence also translation-bounded$\, )$}
  measure on\/ $G$. If $\vL$ is
  uniformly discrete, the following statements hold.
\begin{enumerate}\itemsep=2pt 
\item The set\/
  $ B_\varepsilon \defeq \{ x \in \vL : \mathrm{Re} \bigl( a (x)
  \bigr) \geqslant a (0) - \varepsilon \}$ is relatively dense for
  every\/ $\varepsilon >0$.
\item The measure\/ $\mu$ is sup-almost periodic.
\item The measure\/ $\mu$ is norm-almost periodic.  
\item There is a CPS\/ $(G,H,\cL)$ and some\/ $h \in C^{}_{0} (H)$
  such that\/ $\mu = \omega^{}_{h}$.
\end{enumerate}   
\end{theorem}

\begin{proof}
  {First, let us note that, since $\mu$ is a positive
    definite, {pure point} measure, the function $a$ on $G$
    given by $a (x)= \mu (\{x \})$ is positive definite and supported
    inside $\vL$ by \cite[Prop.~2.4]{LS1}.  Therefore, we have
    $a (0) \in \RR$, and $\left| a (x) \right| \leqslant a (0)$ holds
    for all $ x \in \vL$.  Consequently,\footnote{{Note
        that, since $\vL$ is uniformly discrete and the coefficients
        $a(x)$ are bounded by $a(0)$, we also obtain the
        translation-boundedness of $\mu$ directly from
        positive definiteness, without reference to strong almost
        periodicity.}} {we also have $0\in\vL$ (since
      $\mu\neq 0$) and}
    $\big\lvert \text{Re} \bigl(a (x) \bigr) \big\rvert \leqslant a
    (0)$ for all $x \in \vL$.}

  Next, as $\vL$ is uniformly discrete, we can find a relatively
  compact open neighbourhood $U$ of $0$ such that
  $(x+U) \cap (y+U) \neq \varnothing$ for any $x,y \in \vL$ implies
  $x=y$.

  Let {$g\in C_{\mathsf{c}} (G)$} be a function with values
  in $[0,1]$ such that $g(0) =1$, $g(-x)=g(x)$ and
  $\supp (g) \subseteq U$.  The previous property implies that any
  translate of $\supp (g)$ meets $\vL$ in at most one point.
  Therefore, if $\bigl( \delta^{}_{\! \vL} *g \bigr) (z) \neq 0$ for
  some $z \in G$, there exists a point $y \in \vL$ so that
\begin{equation}\label{eq:ap-1}
    \bigl( \mu *g \bigr) (z) \, = \, g (z-y) \, a (y) \ts .
\end{equation}
As $\mu$ is strongly almost periodic, $\mu*g$ is uniformly (or Bohr)
almost periodic, and the set
\[
   V_{\varepsilon} \, \defeq \, \big\{ t \in G :  \| T^{}_t \ts \mu *g
       - \mu*g \|^{}_\infty < \varepsilon \big\}
\]
is relatively dense for each $\varepsilon > 0$. We claim that
\[
   V_{\varepsilon} \, \subseteq \, B_\varepsilon+U  ,
\]
{with $B_{\varepsilon}$ as defined in statement (1)
  of the theorem.}
As $U$ is a relatively compact subset of $G$, once we prove this
claim, the relative denseness of $V_{\varepsilon}$ implies the
relative denseness of $B_\varepsilon$.

To proceed, let $t \in V_{\varepsilon}$. Then, as
$\| T^{}_t \ts \mu *g - \mu*g \|^{}_\infty < \varepsilon$, we
certainly have the inequality
$ \left| \bigl( T^{}_t \ts \mu *g \bigr) (t) - \bigl( \mu*g \bigr) (t)
\right| < \varepsilon $. By Eq.~\eqref{eq:conv-2}, this is equivalent
to
\[
   \left| \bigl( \mu *g \bigr) (0) - \bigl( \mu*g \bigr) (t)
    \right| \, < \,  \varepsilon \ts .
\]
Since $B_{\varepsilon}\subseteq B_{\varepsilon'}$ for
$\varepsilon\leqslant\varepsilon'$, it suffices to show relative
denseness of the sets $B_{\varepsilon}$ for all sufficiently small
$\varepsilon>0$. Thus, we now assume
\begin{equation}\label{eq:ap-2}
    \bigl( \mu*g \bigr) (0) \, = \, a (0)
     \, > \, \varepsilon \, > \, 0
\end{equation}
and hence $\bigl( \mu*g \bigr) (t) \neq 0$. Therefore, by
Eq.~\eqref{eq:ap-1}, there exists a unique $y \in \vL$ so that
\begin{equation}\label{eq:ap-3}
   \bigl( \mu *g \bigr) (t) \, = \, g(t-y) \,  a (y) \ts .
\end{equation}
With the relations from Eqs.~\eqref{eq:ap-2} and \eqref{eq:ap-3}, we
obtain the following chain of implications,
\[
\begin{split}
  & \left| \bigl( \mu *g \bigr) (0) - \bigl( \mu*g \bigr) (t) \right|
  \,  < \, \varepsilon  \quad \Rightarrow  \quad
 \left| \text{Re} \bigl( a (0) -
   g(t-y) \ts a(y) \bigr) \right| \, < \, \varepsilon \\[1mm]
    \Rightarrow \quad
  & \left| a (0) -  \text{Re} 
  \bigl(g(t-y) \ts a (y) \bigr) \right| \, < \, \varepsilon 
  \quad \Rightarrow \quad  \text{Re} \bigl(g(t-y) \ts a (y) \bigr)
  \, \geqslant \,   a (0) - \varepsilon \, > \, 0 \\[1mm]
  \Rightarrow  \quad & \, \text{Re} \bigl( a (y) \bigr) 
  \, \geqslant \, g(t-y) \, \text{Re} \bigl( a (y)  \bigr) 
  \, \geqslant  \, a (0) -  \varepsilon  \ts ,
\end{split}
\]
where the last step follows from $0 \leqslant g(t-y) \leqslant 1$ in
conjunction with the observation that
$g(t-y) \ts a (y)= \bigl( \mu *g \bigr) (t) \neq 0$.  This implies
$y \in B_\varepsilon$ together with $t-y \in \supp (g) \subseteq
U$. Thus
\[
   t \, = \, y+(t-y) \in  B_\varepsilon +U  \ts ,
\]
and we are done with the first claim.

Next, consider claim (2). As $a$ is positive definite, Krein's
inequality \cite[Cor.~2.5]{LS1} implies
\[
    \left| a (x+t) - a (x) \right|^2  \, \leqslant \,
    2 \ts a (0) [ a (0)- \text{Re} ( a (t))] \ts .
\]
Therefore, one has
\[
    \big\| \mu- T^{}_t \ts \mu \big\|^{2}_\infty \, \leqslant \,
     2 \ts a (0) [ a (0)- \text{Re}( a (t))] \ts .
\]
This shows that
\[
   B_{\!\frac{\varepsilon^2}{2 \ts a (0)}} \, 
   \subseteq \, P_\varepsilon \ts ,
\]
where the $P_\varepsilon$ are the sets of $\varepsilon$-almost periods
which define sup-almost periodicity.

By \cite[Lemma~5.3.6]{NS11}, sup-almost periodicity implies
norm-almost periodicity in this case, as $\vL$ is uniformly discrete.
Finally, claim (4) is \cite[Thm.~5.4.2]{NS11}; a stronger version will
be given in Theorem~\ref{thm:T2} below.
\end{proof}

{Recall that positive definite measures are Fourier
  transformable \cite[Thm.~4.5]{BF}, and that strong almost
  periodicity then implies the Fourier transform to be a pure point
  measure; see \cite[Cor.~4.10.13]{ST}.} Thus, under the conditions of
Theorem~\ref{thm:ap}, $\mu$ and $\widehat{\mu}$ are both pure point.
However, for $\mu$ to be doubly sparse, we need to add a condition on
the support of $\widehat{\mu}$.

Indeed, the autocorrelation measure of the Fibonacci chain, see
\cite{TAO1} for details, provides an example of a positive definite
$\ts\SAP$-measure, $\mu\in\cM^{\infty} (\RR)$, with Meyer set support
such that $\widehat{\mu}$ is a positive, pure point measure on $\RR$
with \emph{dense} support, and the same situation applies to the
autocorrelation measures of {aperiodic}
regular model sets in general; see also
\cite{Ric03} for some interesting extensions {beyond bounded
windows.}

At this point, it seems worthwhile to state the following improvement
of Theorem~\ref{thm:ap} for the case that the support of $\mu$ is FLC,
{ where we do not need second countability of $G$. We
refer to \cite{KL,Fav} for Dirac combs with Delone set support.}

\begin{theorem}\label{thm:ap-2}
  Let\/ $\mu$ be a positive definite, pure point measure on
  {the metrisable LCAG\/ $G$, and assume that\/
  $\mu$ has} FLC support and sparse FBS. Then, one has
\[
  \mu \, = \sum_{i=1}^{N} \sum_{x\in \vG + \tau^{}_{i}}
  P^{}_{i} (x) \ts \delta_{x}
  \quad\text{and}\quad
  \widehat{\mu} \, = \sum_{j=1}^{M}
  \sum_{y\in \vG^{\ts 0} + \sigma^{}_{j}}
  Q^{}_{j} (y) \ts \delta_{y}
\]
for some lattice\/ $\vG\subseteq G$ and some trigonometric
polynomials\/ $P_{i}$ on\/ $G$ and\/ $Q_{j}$ on\/ $\widehat{G}$.
\end{theorem}

\begin{proof}
  {Recall first that $\supp(\mu)$ being FLC means that
  $\supp(\mu)-\supp(\mu)$ is locally finite.  Then, since $\mu$ is
  positive definite, we have $|\mu(\{x\})|\leq \mu(\{0\})$ for all
  $x\in G$ as in the proof of Theorem~\ref{thm:ap}, and $\mu$ is thus
  translation bounded.  The assumption that $\mu$ has a sparse FBS
  implies that $\mu$ is transformable as a measure and that
  $\widehat{\mu}$ is pure point, hence $\mu$ is strongly almost
  periodic.}

  Now, by Theorem~\ref{thm:ap}, $\mu$ is norm-almost periodic.
  Invoking the implication (i) $\Rightarrow$ (vi) from
  \cite[Thm.~5.5.2]{NS11}, we see that there exists a CPS $(G,H,\cL)$
  and a function $h\in C_{\mathsf{c}} (H)$ such that
  $\mu = \omega^{}_{h}$. In particular,
  $\supp (\mu) \subseteq \oplam \bigl(\supp (h)\bigr)$, so
  $\supp (\mu)$ is a subset of a Meyer set.

  The claim now follows from Theorem~\ref{thm:trig-coeffs}.
\end{proof}

\subsection{Doubly sparse sup-almost periodic measures}

Our aim here is to characterise positive definite measures $\mu$ with
uniformly discrete support and sparse FBS.  The key to this
characterisation is the sup-almost periodicity of such a measure as
obtained above.

In fact, given a sup-almost periodic measure $\mu$, our results
require only weak uniform discreteness of its support.  In line with
Theorem~\ref{thm:ap}, note that, when the support of $\mu$ is weakly
uniformly discrete, $\mu$ is sup-almost periodic if and only if it is
norm-almost periodic \cite[Lemma~5.3.6]{NS11}.  This means that all
measures we consider in this section are actually norm-almost
periodic.

\begin{theorem}\label{thm:T2} 
  Let\/ $0\neq \mu =\sum_{x \in \vL} a (x) \, \delta_x$ be a
  translation-bounded, sup-almost periodic measure on\/ $G$. If\/
  $\mu$ has weakly uniformly discrete support, $\vL$, and sparse FBS,
  there is a CPS\/ $(G, H, \cL)$ and some\/ $h \in \Cz(H)$ such that
\begin{enumerate}\itemsep=2pt
  \item $\mu= \omega^{}_h$;
  \item $h \in L^1(H)$ with support of finite measure;
  \item $\widehat{\mu}=\dens(\cL) \, \omega^{}_{\widecheck{h}}$;
  \item $\widecheck{h} \in L^1 \bigl(\widehat{H}\bigr) \cap
    C^{}_{0} (\widehat{H})$, with support of finite measure;
  \item
     ${H}$ has an open and closed compact subgroup, $\KK$;
  \item $\widehat{H}$ has an open and closed compact subgroup.
\end{enumerate}
\end{theorem}

\begin{proof} 
  Since $\mu$ is sup-almost periodic, \cite[Thm.~5.4.2]{NS11} implies
  the existence of a CPS\/ $(G, H, \cL)$, with the group $H$
  metrisable, and that of a function\/ $h \in \Cz(H)$ such that\/
  $\mu = \omega^{}_h$. Let\/ $U=\{y\in H : h(y)\neq 0\}$ as before.
  Since $\supp(\mu)$ is weakly uniformly discrete, $\theta_H(U)$ is
  finite, by Corollary~\ref{coro:cor1}.  Then, $h\in L^1(H)$ by Fact
  \ref{fact:h-is-L1}, and claims $(1)$ and $(2)$ are verified.
    
  Let $(\widehat{G}, \widehat{H}, \cL^0 )$ be the dual CPS.  To show
  claim $(3)$, we use a function of compact support to construct a
  measure whose Fourier--Bohr coefficients are `close' to those of
  $\mu$.  Fix a van Hove sequence $\cA$ in $G$, set
  $ d = \udens^{}_{\cA} (\vL) $, which is finite because $\vL$ is
  weakly uniformly discrete, and fix an arbitrary $\varepsilon>0$.
  Since $h\in\Cz(H)$, there exists a compact set $K^{}_0\subseteq H$
  such that
\[
      |h(y)| \, < \, 
      \myfrac{\varepsilon}{d+\dens(\cL) \ts\ts \theta_{\nts H}(U)}
      \, \eqdef \, \varepsilon^{}_1
\]
holds for every $y\not\in K^{}_0$.  We may choose a relatively compact
open set $V\supseteq K^{}_0$ and an $f\in\Cc(H)$ such that
$1_{K_0}\leqslant f \leqslant 1_V$. Then, setting $g\defeq f\ts h$, we
have
\[   
      | g(y)| \, \leqslant \, | h(y)| \quad \text{for all }  y\in H .
\]
Further, for $y\in K^{}_0$, we have $h(y)=g(y)$ and, for
$y\not\in K^{}_0$,
\[ 
       |h(y)-g(y)| \, = \, |h(y)(1-f(y))|
       \, \leqslant \, |h(y)| \, < \,
       \varepsilon^{}_1 \ts . 
\]
   Consequently, we get
\[
       \|h-g\|^{}_\infty \, < \, \varepsilon^{}_1 \ts . 
\]
Now, consider the measure $\omega_g$. In general, we may not assume
that $\omega_g$ is transformable but, since $g\in\Cc(H)$, $\omega_g$
is strongly almost periodic by \cite[Thm.~3.1]{LR1}.  Then, by
Definition \ref{def:FB-coeffs}, we may consider the Fourier--Bohr
coefficients of $\omega_g$,
\[
     c_{\chi}(\omega_g) \, 
      \, = \lim_{n\to\infty}\frac{\bigl( \overline{\chi}\ts 
      \omega_g \bigr) (A_n)}{\vol (A_n)} \ts ,
\]
for each $\chi\in\widehat{G}$. Now,
   \cite[Thm.~3.3]{LR1}  implies that
\[ 
  c_{\chi}(\omega_g) \, = \, \dens(\cL) \int_H \chi^{\star}(t) \ts
  g(t) \dd t \, = \, \dens(\cL) \, \widecheck{g} (\chi^{\star})
\]
for all $\chi\in \pi_{\widehat{G}}(\cL^0)$, and $c_{\chi}(\omega_g)=0$
otherwise.

Since $|g|\leqslant |h|$, we have $\supp(\omega_g)\subseteq\vL$, while
$\|h-g\|^{}_{\infty}<\varepsilon^{}_1$ implies
$|\omega_h -\omega_g|(t) <\varepsilon^{}_1$ for all $t\in
\vL$. Finally, since $\mu$ is Fourier transformable, we have
$\widehat{\mu}(\{\chi\})=\MM(\overline{\chi}\mu)$ for all
$\chi\in\widehat{G}$, by an application of \cite[Prop.~3.14]{CRS}.
For all $\chi\in \widehat{G}$, this gives us
\[ 
   \begin{split}
         \bigl| \widehat{\mu}(\{\chi\}) - c_{\chi}(\omega_g) \bigr| \,
       & = \lim_{n\to\infty} \frac{1}{\vol(A_n)} \left| 
       \int_{A_n} \chi(t) \dd(\mu - \omega_g) \right| \\[2mm]
        &\leqslant \, \limsup_{n\to\infty}  \frac{1}{\vol(A_n)} 
        \sum_{t\in \vL\cap A_n} \bigl| \chi(t)\bigl(
           \omega_h(\{t\})-\omega_g(\{t\})\bigr) \bigr| \\[1mm]
        &< \, \limsup_{n\to\infty} \frac{1}{\vol(A_n)} 
          \sum_{t\in \vL\cap A_n} \! \varepsilon^{}_1 
         \; = \; d \ts\ts \varepsilon^{}_1 \ts .
     \end{split}
 \]
 For $\chi\not\in \pi_{\widehat{G}}(\cL^0)$, since $c_{\chi}(w_g)=0$
 (and $d\ts \varepsilon^{}_1$ is a product of fixed constants with an
 arbitrary $\varepsilon>0$, so that
 $d \ts \varepsilon^{}_{1} = \cO (\varepsilon)$), we thus have
\begin{equation}\label{eq:notsupp}
    \widehat{\mu}(\{\chi\}) \, = \, 0 \ts ,
    \quad \text{for all }
     \chi\not\in \pi_{\widehat{G}}(\cL^0)\ts .
\end{equation}
Let $\chi\in \pi_{\widehat{G}}(\cL^0)$. Then, from above,
\[
    \left|\widehat{\mu}(\{\chi\})-\dens(\cL)\ts\ts
      \widecheck{g}(\chi^{\star})\right|
    \, < \, d\ts\ts \varepsilon^{}_1 \ts .
\]
Now, $|g|\leqslant |h|$ implies $0=h(y)=g(y)$ for all $y\not\in U$,
and thus
\[ 
   \begin{split}
     \left| \widecheck{g}(\chi^{\star}) -
       \widecheck{h}(\chi^{\star})\right| \,
      &=  \left| \int_H \chi^{\star}(y) (g(y)-h(y)) \dd y \right| 
      \, =  \left| \int_U \chi^{\star}(y) (g(y)-h(y))
           \dd y \right| \\[2mm]
      & \leqslant  \int_U | g(y)-h(y) | \dd y 
   			\: <  \: \theta^{}_H (U) \,\varepsilon^{}_1\ts .
   \end{split}
\]
Multiplying this with $\dens (\cL)$ and combining it with the previous
inequality, we obtain
\[ 
   \bigl| \widehat{\mu}(\{\chi\})-\dens(\cL)\,
     \widecheck{h}(\chi^{\star})\bigr| \,
    < \, \bigl(d+\dens(\cL)\, \theta_H(U) \bigr) 
    \,\varepsilon^{}_1  \: = \: \varepsilon \ts ,
\]
and since $\varepsilon >0 $ was arbitrary, we have
\begin{equation}\label{eq:supp}
     \widehat{\mu}(\{\chi\}) \, = \, \dens(\cL)
      \, \widecheck{h}(\chi^{\star})
      \quad \text{for all } \chi\in \pi_{\widehat{G}}(\cL^0) \ts .
\end{equation}
Combining \eqref{eq:notsupp} and \eqref{eq:supp} gives claim $(3)$,
and as $\mu$ has sparse FBS, claim $(4)$ is direct from
Corollary~\ref{coro:cor1} (and \cite[Thm.~1.2.4]{Rudin}).
   	
To see claim $(5)$, we reason as in Remark~\ref{rem QUP}.  Since both
$h$ and $\widecheck{h}$ have finite measure support, the QUP fails for
$H$. Then, by \cite[Thm.~1]{HOG}, the identity component of $H$ must
be compact.  Recall that $H$, as an LCAG, has an open and closed
subgroup of the form $\RR^d \!\times\! \KK$. Since the identity
component of $H$ is compact, we have $d=0$, and hence $H$ has an open
and closed subgroup $\KK$ which is compact.  The same argument applied
to the dual group $\widehat{H}$ verifies claim $(6)$, and we are done.
\end{proof}

\begin{remark}
  Note that we have $\mu = \omega^{}_{h}$ together with
  $\widehat{\mu} = \dens (\cL) \, \omega^{}_{\widecheck{h}}$ in
  Theorem~\ref{thm:T2}, from claims $(1)$ and $(3)$. Let us emphasise
  that this actually is a PSF for the lattice $\cL^{0}$. Indeed, with
\[
  K^{}_{2} (G) \, \defeq \, \mathrm{span}^{}_{\CC}
  \{ f*g : f,g \in C_{\mathsf{c}} (G) \} \ts ,
\]
the second relation means
$\langle \omega^{}_{h}, g \rangle = \dens (\cL) \langle
\omega^{}_{\widecheck{h}} , \widecheck{g}\ts \rangle$ for all
$g \in K^{}_{2} (G)$; see \cite[Sec.~4.9]{MoSt} for background.  By
definition, one has
\[
  \langle \omega^{}_{h}, g \rangle \, = \ts
  \sum_{x \in L} h (x^{\star}) \, \delta^{}_{x}
   (g) \, =  \! \sum_{(x,x^{\star})\in \cL} \!
  g (x) \, h (x^{\star})  \, = \, \delta^{}_{\cL}
  (g \otimes h) \ts ,
\]
while the other side contains
\[
  \big\langle \omega^{}_{\widecheck{h}} , \widecheck{g}\ts \big\rangle
  \, = \! \sum_{(\chi,\chi^{\star}) \in \cL^{0}} \!  \widecheck{h}
  (\chi^{\star}) \, \widecheck{g} (\chi) \, = \, \delta^{}_{\cL^{0}}
  \bigl(\widecheck{g} \otimes \widecheck{h} \bigr) .
\]
Consequently, $\widehat{\omega^{}_{h}} = \dens (\cL) \,
\omega^{}_{\widecheck{h}}$ simply means that, for all
$g\in K^{}_{2} (G)$, we have
\[
  \langle \delta^{}_{\cL}, g \otimes h \rangle \, = \, \dens (\cL) \,
  \big\langle \delta^{}_{\cL^{0}} , \widecheck{g\otimes h} \big\rangle ,
\]
which justifies the interpretation; see \cite{Ric03}
for related results. \exend  
\end{remark}

The last proof, in conjunction with \cite[Thm.~7.6]{LS2}, has a
direct consequence as follows.

\begin{coro}\label{coro:diffract}
  Under the conditions of Theorem~\textnormal{\ref{thm:T2}}, the
  measure\/ $\mu = \omega^{}_{h}$ has a unique autocorrelation
  measure, namely\/
  $\gamma = \omega^{}_{h} \circledast \widetilde{\omega^{}_{h}} =
  \dens (\cL)\ts\ts \omega_{h * \widetilde{h}} \ts $, and the
  corresponding diffraction measure is\/
  $\ts \widehat{\gamma} = \dens (\cL)^2 \ts \omega_{\lvert \check{h}
    \rvert^2}$.  \qed
\end{coro}

\begin{remark} 
  Since $\KK$ is open and closed in $H$, the factor group $H/\KK$ is
  discrete. Therefore,
\[
    \KK^0 \, \defeq \, \{ \chi \in \widehat{H} : 
    \chi \equiv 1 \mbox{ on } \KK \} 
    \, \simeq \, \widehat{H/\KK} 
\]
is closed and compact.  \exend
\end{remark}

\begin{prop}\label{prop5}
  Under the conditions of Theorem~\textnormal{\ref{thm:T2}}, the
  measure\/ $\mu$ may be approximated in any\/ $K$-norm\/
  {$\| . \|^{}_{K}$ for measures} by strongly almost
  periodic measures\/ $\mu_n$ that are supported inside sets\/ $\vG +F_n$,
  where\/ $\vG$ is a lattice in\/ $G$ and the\/ $F_n \subseteq G$ are
  finite.  Moreover, the Fourier--Bohr coefficients of the measures\/
  $\mu_n$ converge to those of\/ $\mu$.
\end{prop}

\begin{proof}
  We employ the setting of the proof of Theorem~\ref{thm:T2}.  Fix a
  compact set $K\subseteq G$ and set $\vG \defeq\oplam(\KK)$.  We will
  now construct an increasing sequence of finite sets $F_n\subseteq G$
  such that $\mu_n = \mu|_{\vG +F_n} \in \ts\SAP(G)$ and
  $\|\mu_n - \mu\|^{}_K \leqslant \tfrac{1}{n}$, {where
    $\|.\|^{}_{K}$ is defined as in \eqref{eq:def-norm}.}

  Since $\vL$ is weakly uniformly discrete, there exists an $N\in\NN$
  such that, for all $t\in G$,
\begin{equation}\label{eq:card}
   \card((t+\vL)\cap K) \, < \, N\ts .
\end{equation}
As usual, let $U=\{ y\in H: h(y)\neq 0\}$, where $h\in \Cz(H)$.  For
each $n\in \NN$, there exists a compact set $W_n\subseteq U$ such that
\begin{equation}\label{eq:htail}
     |h(x)|  \, < \,  \myfrac{1}{nN}
     \quad \text{for all } x\not\in W_n \ts .
\end{equation}
Since $\pi^{}_{H} (\cL)$ is dense in $H$ and $\KK$ is open in $H$, we have
$\pi^{}_{H} (\cL) + \KK = H$.  By the compactness of $W_n$, we can find
a finite set $F_n^{\star}\subseteq \pi^{}_{H} (\cL)$ such that
$W_n\subseteq F_n^{\star} +\KK$. Then, let
\[ 
    F^{}_n  \, \defeq \, \oplam(F_n^{\star})
    \quad \text{and} \quad
     h^{}_n \, \defeq \, h\ts\ts1^{}_{F_n^{\star}+\KK} \ts . 
\]
Since $\KK$ is open and closed in $H$ and $F_n^{\star}$ is finite,
$F_n^{\star}+\KK$ is also open and closed in $H$. Consequently, $h_n$
is continuous. Moreover, as $F_n^{\star} +\KK$ is compact, we have
$h_n\in\Cc(H)$. Setting
\[  
    \mu_n \, \defeq \, \omega^{}_{h_n} \ts ,
\] 
we have $\mu_n \in\SAP(G)$ by \cite[Thm.~3.1]{LR1} and
$\mu_n = \mu |^{}_{\vG +F_n}$ by construction.  Now, \eqref{eq:htail}
ensures that $|h(y)-h_n(y)|<\frac{1}{nN}$ for all $y\in H$ and thus
that $\bigl| \mu(\{x\})-\mu_n(\{x\}) \bigr| <\frac{1}{nN}$ for all
$x\in G$.  It is clear that $\supp(\mu-\mu_n)\subseteq\vL$.
Consequently, via \eqref{eq:card}, we see that
\[ 
   \|\mu - \mu_n\|^{}_K  \, < \, \myfrac{1}{n} \ts .  
\]
Note that the sets $W_n$, and thus $F_n^{\star}$ and $F_n$, may be
chosen to be increasing, as claimed.

Finally, for $\chi\in \widehat{G}$, observe that
\[ 
   \bigl| \widehat{\mu}(\{\chi\}) - c_{\chi}(\mu_n) \bigr|
   \, = \, \MM \bigl( \overline{\chi}(\mu-\mu_n) \bigr)
   \, \leqslant \,  C\ts \|\mu-\mu_n\|^{}_K \ts ,  
\]
for some $C>0$, which verifies the convergence of the Fourier--Bohr
coefficients.
\end{proof}

Comparing the results of this section with those of
Section~\ref{sec:meyer}, we see that, while sup-almost periodicity of
a pure point measure $\mu$ enables its representation as a model comb,
the weight function $h$ has compact support if and only if
$\supp (\mu)$ is FLC (or Meyer); see \cite[Thm.~5.5.2]{NS11}.  This
makes the calculations for doubly sparse sup-almost periodic measures
with only uniformly discrete support a little more delicate.
Nevertheless, we obtain almost everything that we did in
Section~\ref{sec:meyer}, apart from the support of $\mu$ being
crystallographic, and even this we `almost' get.

Now, we combine these results with Theorem~\ref{thm:ap} to obtain the
characterisation for positive definite measures with uniformly
discrete support.

\begin{coro}\label{coro:posdef}
  Let\/ $0\neq \mu =\sum_{x \in \vL} a (x) \, \delta_x$ be a positive
  definite measure with uniformly discrete support, $\vL$, and sparse
  FBS. Then, the conclusions of Theorem~\textnormal{\ref{thm:T2}},
  Corollary~\textnormal{\ref{coro:diffract}} and
  Proposition~\textnormal{\ref{prop5}} hold.
\end{coro}

\begin{proof}
  By the argument used in the proof of Theorem~\ref{thm:ap-2}, $\mu$
  is translation bounded. From Remark~\ref{rem:sparse}, we see that
  $\mu$ is an $\ts\SAP$-measure. Now, Theorem~\ref{thm:ap} implies
  that $\mu$ is sup-almost periodic, and the rest is clear.
\end{proof}

{There is also a Fourier-dual version, which we can
  formulate as follows.

\begin{coro}\label{coro:posdef-dual}
  Let\/ $0\neq \mu =\sum_{x \in \vL} a (x) \, \delta_x$ be a positive
  measure with sparse support, $\vL$. Further, assume that\/ $\mu$ is
  Fourier transformable and that\/ $\supp(\widehat{\mu})$ is uniformly
  discrete.  Then, the conclusions of
  Theorem~\textnormal{\ref{thm:T2}},
  Corollary~\textnormal{\ref{coro:diffract}} and
  Proposition~\textnormal{\ref{prop5}} hold.  \qed
\end{coro}
}

\section{Specific results for $G=\RR^d$}\label{sec:real}

For arguments in $G=\RR^d$, the usual framework is that of tempered
distributions.  We use $\cS (\RR^d)$ and $\cS' (\RR^d)$ to denote the
spaces of Schwartz functions and tempered distributions on $\RR^d$,
respectively, and $\langle T , \varphi \rangle \defeq T (\varphi)$ for
the pairing of a distribution and a test function.  The distributional
Fourier transform of $\mu\in \cS' (\RR^d)$ is the distribution
$\nu\in\cS' (\RR^d)$ such that
$\langle\nu,\varphi \ts \rangle = \bigl\langle\mu,
\widehat{\varphi}\ts \bigr\rangle$ for all test functions
$\varphi\in\cS(\RR^d)$. {Recall that a translation-bounded
  measure on $\RR^d$ is always a tempered distribution, referred to as
  a tempered measure; see \cite{Y} for general background, and
  \cite{ST} for further notions, such as translation-boundedness for
  tempered distributions.}

In previous sections, we have assumed our measures to be translation
bounded and Fourier transformable. The connection between the
distributional Fourier transform and Fourier transformability as an
unbounded Radon measure, as we have considered, was clarified in
\cite{Nicu}, where it was shown that {a measure $\mu$ on
  $\RR^d$ is Fourier transformable as a measure if and only if it is
  tempered and its distributional Fourier transform is a
  translation-bounded measure.}  Thus, in the Euclidean setting, a
measure $\mu$ is translation bounded and Fourier transformable if and
only if its distributional Fourier transform $\nu$ is { a
  translation-bounded and Fourier-transformable measure.} We begin
this section by establishing some sufficient conditions for
transformability and translation-boundedness.

In \cite{ST}, the notions of weak and strong almost periodicity for
tempered distributions were defined, and it was shown that these
definitions coincide with the classical ones for the class of
translation-bounded measures on $\RR^d$.

\begin{lemma}\label{lem:wap-td-is-tb}
  Let\/ $\mu\in\cS'(\RR^d)$ be a measure with uniformly discrete
  support that is weakly almost periodic as a tempered
  distribution. Then, $\mu$ is a translation-bounded measure and
  thus\/ $\mu\in \WAP(\RR^d)$.
\end{lemma}

\begin{proof}
  By \cite[Rem. 5.1]{ST}, $\mu = \sum_{x\in\vL} a(x) \ts \delta_{x}$
  is translation bounded as a tempered distribution, meaning that
  $\mu\ast f \in C_{\mathsf{u}}(\RR^d)$ for all $f\in\cS(\RR^d)$.
  Now, since $\supp(\mu)$ is uniformly discrete, we may choose an open
  neighbourhood $U$ of $0$ such that $(x+U) \cap (y+U) = \varnothing$
  for all $x,y\in\supp(\mu)$ with $x\neq y$. Select a function
  $f\in C^{\infty}_{\mathsf{c}}(\RR^d)$ such that $\supp(f)\subset U$
  and $f(0)=1$.  Via a simple calculation, one can verify that
\[
     \big\lvert \mu\ast f\ts \big\rvert (x) \, = \,
     \bigl( \lvert \mu \rvert \ast \lvert f \rvert \bigr) (x)
      \, = \, \big\lvert a(x) \big\rvert
\]
holds for all $x\in\supp(\mu)$; compare \cite[Lemma~5.8.3]{NS11}.
Then, there exists a $C>0$ such that
\[ 
  |a(x)| \, \leqslant \, C \quad \textrm{for all} \;\ts x\in\supp(\mu)
\]
and thus, since $\supp(\mu)$ is uniformly discrete, $\mu$ is a
translation-bounded measure. Hence, by \cite[Thm.~5.3]{ST}, $\mu$ is
also weakly almost periodic as a measure.
\end{proof}

\begin{remark}\label{rem:tempered}
  In particular, any tempered measure, $\mu$, whose distributional
  Fourier transform is a measure, is a weakly almost periodic tempered
  distribution \cite[Thm.~5.1]{ST}. Thus, if $\mu$ also has uniformly
  discrete support, the conclusions of Lemma~\ref{lem:wap-td-is-tb}
  hold.  \exend
\end{remark}

The following generalises \cite[Lemma~2]{LO1}, which assumes that both
$\vL = \supp (\mu)$ and $S = \supp (\nu)$ are uniformly discrete, and
\cite[Lemma~3.1]{LO2}, which shows only the translation-boundedness of
$\mu$.

\begin{lemma}\label{lem:FT-able}
  Let $\vL \subset \RR^d$ be uniformly discrete, let\/
  $S\subset \RR^d$ be weakly uniformly discrete, and consider\/
  $\mu, \nu \in \cS' (\RR^d)$ as given by
\begin{equation}\label{eq:mu_nu}
     \mu \, = \sum_{x\in\vL} a(x) \, \delta_x
     \quad \text{and} \quad
     \nu \, = \sum_{y\in S} b(y) \, \delta_y \ts .
\end{equation}    
Now, if\/ $\nu$ is the distributional Fourier transform of\/
$\mu$, then\/ $\mu$ and\/ $\nu$ are translation-bounded measures that
are Fourier transformable as measures and, as such, satisfy\/
$\widehat{\mu} = \nu$.
\end{lemma}

\begin{proof}
  As noted above, the translation-boundedness of $\mu$ is a
  consequence of Lemma~\ref{lem:wap-td-is-tb}.  To see the
  translation-boundedness of $\nu$, it suffices to show that the set
  of coefficients, namely $\{b(y): y\in S\}$, is bounded.

  Let $y\in S$ be arbitrary but fixed, and select
  $c \in C^{\infty}_{\mathsf{c}} (\RR^d)$ with $\widehat{c} (y) = 1$
  and $\int_{\RR^d} \lvert c (x) \rvert \dd x \leqslant 2$, which is
  clearly possible. From \cite[Prop.~4.1]{ST}, we know that the
  function $g=\mu\ast c$ is bounded and uniformly continuous, and thus
  defines a regular tempered distribution. Its distributional Fourier
  transform, $\nu\,\widehat{c} \eqdef\rho$, is a finite measure.  Now,
  by \cite[Thm.~7.2]{ST}, we have
\[
  \rho ( \{ y \} ) \, = \,
  \MM \bigl( \ee^{- 2 \pi \ii y (.)} g \bigr) \ts .
\]
 But this gives us
\[
\begin{split}
  \lvert b(y) \rvert \, & = \lvert\ts \widehat{c}
    (y) \ts b (y) \rvert
  \, = \, \lvert \ts \rho ( \{ y \} ) \rvert \, = \,
   \big\lvert  \MM \bigl( \ee^{-2 \pi \ii y (\cdot)} \ts g  \bigr) 
   \big\rvert \, \leqslant \, \MM \bigl(  \lvert
   \ee^{- 2 \pi \ii y (\cdot)} \ts g \ts \rvert \bigr) 
   \, = \, \MM \bigl( \lvert \mu * c \rvert \bigr) \\[3mm]
   & = \, \lim_{n\to\infty}   \frac{1}{(2 n)^d} \int_{[-n,n]^d}
     \lvert \mu * c \rvert (y) \dd y \: \leqslant \:
     \limsup_{n\to\infty} \frac{1}{(2 n)^d} \int_{[-n,n]^d}
     \bigl(\lvert \mu\rvert * \lvert c \rvert \bigr) (y) \dd y \\[2mm]
   & = \, \limsup_{n\to\infty} \frac{1}{(2 n)^d} \int_{[-n,n]^d}
   \int_{\RR^d} \lvert c (y-t) \rvert \dd \lvert \mu \rvert (t) \dd y \ts .  
\end{split}
\]

Since $c$ has compact support, there exists some $m$ such that
$\supp (c) \subseteq [-m,m]^d$. Then, for all $y\in [-n,n]^d$, we have
$c (y-t) =0$ outside of $[-n-m, n+m]^d$. Via Fubini, we thus get
\[
\begin{split}
  \lvert b(y) \rvert \, & \leqslant \, \limsup_{n}
  \frac{1}{(2 n)^d} \int_{[-n,n]^d} \int_{[-n-m,n+m]^d}
  \lvert c (y-t) \rvert \dd \lvert \mu \rvert (t) \dd y \\[2mm]
  & =\, \limsup_{n}
  \frac{1}{(2 n)^d} \int_{[-n-m,n+m]^d} \int_{[-n,n]^d}
  \lvert c (y-t) \rvert \dd  y \dd \lvert \mu \rvert ( t)  \\[2mm]
  & \leqslant \, \limsup_{n}
  \frac{1}{(2 n)^d} \int_{[-n-m,n+m]^d} \int_{\RR^d}
  \lvert c (y-t) \rvert \dd  y \dd \lvert \mu \rvert ( t)  \\[2mm]
  & \leqslant \, \limsup_{n}
  \frac{2}{(2 n)^d} \lvert \mu \rvert \bigl( [-n-m, n+m]^d \bigr)
  \, \eqdef \, C \ts .
\end{split}
\]

Since $\mu$ is translation bounded, we have $C < \infty$ and hence
$\lvert b(y) \rvert \leqslant C$ for all $y\in S$, which proves that
the set of coefficients is indeed bounded. Finally, since $S$ is
weakly uniformly discrete, it follows that $\nu$ is translation
bounded.

Now, since the measure $\mu$ is tempered and its Fourier transform as
a tempered distribution is a translation-bounded measure, $\mu$ is
Fourier transformable as a measure, by \cite[Thm.~5.2]{Nicu}. The same
statements hold for the measure $\nu$, and we have $\widehat{\mu}=\nu$
as claimed.
\end{proof}

{\begin{remark}

  An interesting question in the context of Lemma~\ref{lem:FT-able}
  is whether one could relax the condition of uniform discreteness
  of $\vL$ to weak uniform discreteness, which would strengthen some
  results in this section. At present, we do not have an answer to
  this question. Also, if true, a proof would need other methods.
  Indeed, already Lemma~\ref{lem:wap-td-is-tb} fails for measures
  with weakly uniformly discrete support, as
\[
  \mu \, \defeq \sum_{k=1}^{\infty} k \,
  \bigl( \delta^{}_{k + \frac{1}{k}} +
     \delta^{}_{k - \frac{1}{k}} - 2 \ts \delta^{}_{k}\bigr)
\]
  clearly demonstrates. \exend
\end{remark}}

Recalling from Section~\ref{sec:sap} that positive definite Radon
measures are Fourier transformable, and again using
\cite[Thm.~5.2]{Nicu}, we summarise some useful sufficient conditions
as follows.

\begin{coro}\label{coro:FTable}
  Let\/ $\mu$ be a tempered measure on\/ $\RR^d$ such that its
  distributional Fourier transform, $\nu$, is a measure. Under any of
  the following conditions, $\mu$ is translation bounded and Fourier
  transformable $($and thus so is\/ $\nu = \widehat{\mu}\ts )\nts :$
\begin{enumerate}\itemsep=2pt
\item $\supp(\mu)$ is uniformly discrete and\/ $\supp(\nu)$ is
  weakly uniformly discrete;
\item $\supp(\mu)$ is uniformly discrete and $\mu$ is positive
  definite;
\item $\supp(\mu)$ is uniformly discrete and $\nu$ is
  translation bounded;
\item $\mu$ is translation bounded and positive definite;
\item $\mu$ and $\nu$ are both translation bounded.  \qed
\end{enumerate}
\end{coro}

Note that a pure point measure may be translation bounded and still
have dense support, {such as the diffraction measure of the
  Fibonacci chain \cite[Sec.~9.4.1]{TAO1}.}  Consequently, sparseness
conditions on the support of a measure can be sufficient but are not
necessary for translation-boundedness.

{At this point, we can harvest Lemma~\ref{lem:FT-able} 
to state the following slightly stronger version of 
Theorem~\ref{thm:trig-coeffs} for the specific case  $G= \RR^d$.

\begin{theorem}
  Let\/ $0 \ne \mu \in \cM(\RR^d)$ be a tempered measure such that\/
  $\supp(\mu)$ is contained in a Meyer set, and assume that the 
  distributional Fourier transform\/ $\nu$ of\/ $\mu$ is a measure whose 
  support is\/ $\cB$-sparse for some van Hove sequence\/ $\cB$ in\/
  $\RR^d$. Then, there is a lattice\/ $\vG \subset \RR^d$ together with
  elements\/ $\tau^{}_1, \ldots , \tau^{}_N, \sigma^{}_1, \ldots , 
  \sigma^{}_M \in \RR^d$ and trigonometric polynomials\/ $P^{}_{i}$
  and\/ $Q^{}_{j}$ on $\RR^d$ such that
\[
    \pushQED{\qed}
    \mu \, = \sum_{i=1}^N \, \sum_{x \in \vG+ \tau^{}_{i}} \!
     P^{}_{i} (x) \, \delta^{}_{x} \qquad \text{and} \qquad
    \widehat{\mu} \, = \sum_{j=1}^{M} \,
    \sum_{y \in \vG^{\ts 0}+ \sigma^{}_j} \! \! Q^{}_{j} (y) \, 
    \delta^{}_{y} \ts .  \qedhere \popQED
\]
\end{theorem}
}

The above results allow us to use the results of
Section~\ref{sec:meyer} in considering a question posed by Meyer
\cite{Meyer}, namely whether there exists a pair of tempered measures
$\mu,\nu$ on $\RR^d$, defined as in \eqref{eq:mu_nu}, such that $\nu$
is the distributional Fourier transform of $\mu$, $\vL$ is a fully
Euclidean model set, and $S$ is locally finite.

\begin{remark} 
  Recall that Meyer's definition of a model set in the context of this
  question requires that the internal space be $H =\RR^n$. As stated in
  Section~\ref{sec:CPS}, we always refer to a CPS of the form
  $(\RR^d,\RR^n,\cL)$ as a \emph{fully Euclidean CPS}.  \exend
\end{remark}

We require one further result as follows.

\begin{lemma}\label{lem:m=0} 
  Let\/ $(\RR^d, \RR^n, \cL)$ be a fully Euclidean CPS, and let\/
  $\oplam(W)$ be a model set in this CPS. If there exists a lattice\/
  $\vG \subset \RR^d$ and a finite set\/ $F \subset \RR^d$ such that
\[
     \oplam(W) \, \subseteq \, \vG+F\ts ,
\]
then\/ $n=0$ and\/ $\oplam(W)$ is a lattice in\/ $\RR^d$.
\end{lemma}

\begin{proof}
  Suppose that such sets $\vG,F\subset \RR^d$ exist. We first show
  that we can choose them such that
  $\vG,F\subseteq L= \pi^{}_G (\cL)$.  Note that 
  \cite[Lemma~5.5.1]{NS11} implies 
  $\vG+F \subseteq \oplam(W) + F^{}_0$ for some
  finite set $F^{}_0$. Then, we get
\[
     \vG \, \subseteq \, \vG + F - F \, = \,
      \oplam(W) +F'
\]
with $F' = F^{}_0 -F$, which is a finite set.  As $\vG$ is a lattice
in $\RR^d$, there are vectors $v_1,\ldots , v_d \in \RR^d$ such that
$\vG = \ZZ\ts v^{}_{1} \oplus \dots \oplus \ZZ \ts v^{}_{d}$.  For any
fixed $j\in\{1,\ldots, d\}$, we have
$ m \ts v_j \in \vG \subseteq \oplam(W)+F'$ for every $m\in\NN$.  As
$F'$ is finite, there exist positive integers $n^{}_1\neq n^{}_2 $ and
a $t \in F'$ such that $ n^{}_1 \ts v_j$ and $ n^{}_2 \ts v_j$ lie in
$\oplam(W)+t$.  Thus, we have
\[ 
   (n^{}_1 - n^{}_2) \ts v_j \, \in 
   \oplam(W) - \oplam(W) \, \subseteq \, L \ts .
\]
In this way, for each $j\in\{1,\ldots, d\}$, we find some
$m_j\in\NN_0$ such that $m_j \ts v_j\in L$.  Setting
$\ell = \textrm{lcm}(m^{}_1,\ldots , m^{}_d)$, we have
\[
    \ell \vG \, \subseteq \, L \ts .
\]
Now, as $\ell \vG\subseteq \vG$ has finite index, there exists a finite
set $J\subset\RR^d$ such that $\vG\subseteq \ell \vG+J$, and we get
\[ 
    \oplam(W) \, \subseteq \, \ell \vG + F+J  \ts .
\] 
Define $F'' = (F+J)\cap L$.  For $x\in \oplam(W)$, there exist
$y\in \ell \vG$ and $z\in (F+J)$ such that $x=y+z$. But as $x\in L$
and $y\in L$, we must have $z\in L\cap (F+J)$, Consequently,
\[ 
   \oplam(W) \, \subseteq \, \ell \vG + F'' \ts ,
\]
where the lattice $\ell \vG$ and the finite set $F''$ are both contained
in $L$.

To continue, we relabel so that, w.l.o.g., $\oplam(W) \subseteq \vG+F$
with $\vG,F\subseteq L$.  Now, invoking \cite[Lemma~5.5.1]{NS11},
there exists a finite set, $F^{}_1\subset \RR^d$, such that
\[
    \vG+F \, \subseteq \, \oplam(W)+F^{}_1 \ts ,
\]
and since $\vG,F,\oplam(W)\subseteq L$, we may as above choose 
$F^{}_1$ such that $F^{}_1 \subset L$.  Then,
\[
    \vG \, \subseteq \, \oplam(W) + F^{}_1 - F 
    \, = \, \oplam(W) - F^{}_2 \ts ,
\]
with a finite set $F^{}_2 \subseteq L$, and thus
\[
    \vG \, \subseteq \, \oplam(W + F_2^{\star}) \ts .
\]
Define $Z\defeq\{x^{\star} : x\in \vG\} = \vG^{\star}$. Then, $Z$ is a
subgroup of $\RR^n$, and so is its closure, $\overline{Z}$. Since
$\overline{Z}\subseteq \overline{W} + F_2^{\star}$, we see that
$\overline{Z}$ is a compact subgroup of $\RR^n$, so we must have
$\overline{Z} = \{0\}$.

Now, recalling that $\oplam(W)\subseteq \vG+F$, we have
$W\subseteq \overline{Z} + F = F$, so $W$ is finite.  But $W$ has
non-empty interior, so we must have $n=0$.  To conclude, we note that,
since $W\subseteq \{0\} = \RR^0$, we have $W=\{0\}$. Hence,
$\oplam(W)$ is a subgroup of $\RR^d$ and thus is a lattice.
\end{proof}

By combining Theorem~\ref{thm:conj1} with Lemma~\ref{lem:m=0} and
Corollary~\ref{coro:FTable}, we can answer a weaker version of Meyer's
question.  Recall that a sparse point set (precisely, $\cB$-sparse for
some van Hove sequence $\cB$) is necessarily locally finite.

\begin{coro}\label{coro:mey}
  There is no {tempered} measure\/
  $0\ne\mu=\sum_{\lambda \in \vL} a(\lambda)\, \delta_\lambda$
  supported inside a model set\/ $\vL\subset \RR^d$ in a non-trivial,
  fully Euclidean CPS\/ $(\RR^d ,\RR^n , \cL)$ such that the
  distributional Fourier transform\/ $\nu$ is a translation-bounded
  measure with sparse support.
\end{coro}

\begin{proof}
  Suppose to the contrary that such a measure, $\mu\ne 0$, exists. By
  Corollary~\ref{coro:FTable}, $\mu$ is translation bounded and
  Fourier transformable as a measure, with measure Fourier transform
  $\nu$. Then, by Theorem~\ref{thm:conj1}, $\supp(\mu)$ is a subset of
  $\vG+F$, for some lattice $\vG$ and $F\subset \RR^d$ finite.

  Next, as $\mu \in \SAP (\RR^{d})$, $\supp(\mu)$ is relatively
  dense. Therefore, by \cite[Lemma~5.5.1]{NS11}, there exists a finite
  set $F' \subset \RR^d$ such that $\vL \subseteq \supp(\mu)+F'$. 
  This implies
\[
     \vL \, \subseteq \, L+(F+F') \ts ,
\]
so, by Lemma~\ref{lem:m=0}, the CPS has internal space $\RR^0=\{0\}$.
\end{proof}

Note that translation-boundedness of $\nu$ in the above result may be
replaced by any of the sufficient conditions in
Corollary~\ref{coro:FTable}.  In fact, a result of Lev and Olevskii
allows us to answer Meyer's question in a little more generality,
namely for the case that $\nu$ is a slowly increasing measure.  Recall
that a tempered measure $\nu$ is \emph{slowly increasing} when
$\lvert \nu \rvert (B_r) = \cO (r^n)$ as $r\to\infty$ for some
$n\in\NN$, {where $B_r$ denotes the ball of radius
$r$ around $0$,}
which is a mild restriction when $\nu$ is a signed or
complex measure.

For slowly increasing measures $\mu$ and $\nu$, defined as in
\eqref{eq:mu_nu}, with $\nu$ the distributional Fourier transform of
$\mu$ and $\supp (\mu) = \vL$ inside a Meyer set, \cite[Thm.~7.1]{LO2}
states that $S = \supp (\nu)$ is either uniformly discrete or has a
relatively dense set of accumulation points. This means that local
finiteness of $S$ forces $S$ to be uniformly discrete in this case
(and Corollary~\ref{coro:FTable} then implies that $\mu$ is
translation bounded and transformable, so we may proceed as
above). The result is also implied by \cite[Thm.~2.3]{LO2}, which is
an $\RR^d$-version of our Theorem~\ref{thm:trig-coeffs}.

\begin{coro}
  Let\/ $\mu=\sum_{\lambda \in \vL} a(\lambda) \, \delta^{}_\lambda$
  be {tempered and} supported inside a model set\/ $\vL$ in
  a non-trivial, fully Euclidean CPS and let\/ $\nu$, the
  distributional Fourier transform of\/ $\mu$, be a slowly increasing
  measure. Then, if\/ $\nu_{\mathsf{pp}}$ has locally finite support,
  it must be trivial, $\nu_{\mathsf{pp}}=0$.
\end{coro}

\begin{proof}
  If $\mu$ is supported inside a fully Euclidean model set, then so is
  $\mu_{\mathsf{s}}$.  By \cite[Thm.~6.1]{ST}, $\nu_{\mathsf{pp}}$ is
  the distributional Fourier transform of $\mu_{\mathsf{s}}$, so
  applying Corollary~\ref{coro:mey} and the comments following it to
  $\mu_{\mathsf{s}}$ gives the result.
\end{proof}

Lemma~\ref{lem:wap-td-is-tb} allows us to use some properties of
weakly almost periodic measures, compare \cite{LS2}, to make some
general statements about the diffraction of measures on $\RR^d$ that
have uniformly discrete support.  The following generalises
\cite[Lemma~10.5]{LO2}, {where we employ the Fourier--Bohr
  coefficients of a measure $\mu$, denoted by $c^{}_{\chi} (\mu)$,
  from Definition~\ref{def:FB-coeffs}.}

\begin{prop}\label{prop:distri}
  Let\/ $\mu$ be a translation-bounded measure on\/ $\RR^d$ such that
  its distributional Fourier transform, denoted by\/ $\nu$, is also a
  measure, and let\/
  $S\defeq \{ \chi \in \RR^d : c^{}_{\chi} (\mu) \ne 0 \}$.  Then, one
  has the following properties:
\begin{enumerate}\itemsep=2pt
\item the autocorrelation\/ $\gamma$ of\/ $\mu$ is unique;
\item $\mu$ possesses the pure point diffraction measure\/
   $ \widehat{\gamma}  =
    \sum_{\chi\in S} |c_{\chi}(\mu)|^2 \, \delta_{\chi}$;
\item $\supp(\nu_{\mathsf{pp}})=S$.
\end{enumerate}
\end{prop}

\begin{proof}
  Via Lemma~\ref{lem:wap-td-is-tb} and Remark~\ref{rem:tempered}, we
  see that $\mu \in \WAP (\RR^d)$.  Then, claims $(1)$ and $(2)$ are
  clear from \cite[Thm.~7.6]{LS2}, while claim $(3)$ now follows from
  \cite[Thm.~7.2]{ST}; compare Definition~\ref{def:FB-coeffs}.
\end{proof}

The following result generalises \cite[Thm.~10.4]{LO2}, since we do
not require the measure $\mu$ to be translation bounded.

\begin{theorem}\label{thm:finally}
  Let \/$\mu$ be a tempered measure that is supported in a Meyer set
  and has an autocorrelation, $\gamma$. The support of the pure point
  part of the diffraction,
  $S\defeq \supp(\widehat{\gamma}^{}_{\mathsf{pp}})$, is either
  uniformly discrete and contained in finitely many translates of a
  lattice, or is not locally finite and has a relatively dense set of
  accumulation points.
\end{theorem}

\begin{proof}
  From \cite[Thm.~5.1]{ST}, $\gamma$ is a weakly almost periodic,
  tempered distribution, so by Lemma~\ref{lem:wap-td-is-tb}, $\gamma$
  is a weakly almost periodic, translation-bounded measure.  Since
  $\gamma$ is supported inside a Meyer set, $\gamma_{\mathsf{s}}$ is
  supported inside a Meyer set as well \cite{NS-weight,Nicu2}.
  
  Noting that $\gamma_{\mathsf{s}}$ and its Fourier transform,
  $\widehat{\ts\ts\gamma_{\mathsf{\ts s}}\ts\ts} =
  \widehat{\gamma}^{}_{\mathsf{pp}}$, are both translation-bounded and
  hence slowly increasing measures, we may apply
  \mbox{\cite[Thm.~7.1]{LO2}} to the measure $\gamma_{\mathsf{s}}$ to
  see that either $S$ has a relatively dense set of accumulation
  points or is uniformly discrete.
   
  The latter case is non-trivial only when $\gamma \ne 0$.  Then,
  observing that $\gamma$ is positive definite and supported inside a
  Meyer set, we see that $\gamma$ is translation bounded and
  transformable by Corollary~\ref{coro:FTable}, so we may apply
  Theorem~\ref{thm:conj1} to $\gamma$ to obtain the result.
\end{proof}

We combine the results of this section as follows.

\begin{coro}\label{coro:almost-done}
  Let\/ $\mu$ be a tempered measure on\/ $\RR^d$ such that its
  distributional Fourier transform\/ $\nu$ is also a measure.  If\/
  $\mu$ is supported inside a Meyer set, $\vL$ say, and if we set\/
  $S\defeq \supp (\nu^{}_{\mathsf{pp}}) \neq \varnothing$, 
  precisely one of the following situations applies:
\begin{enumerate}\itemsep=2pt
\item $S$ contains a relatively dense set of accumulation points;
\item there exists a lattice\/ $\vG$ in\/ $\RR^d$ together with finite
  sets\/ $F, F^{\ts \prime} \subset \RR^d$ such that\\
  $\vL \subseteq \vG + F$ and\/ $S \subseteq \vG^{\ts 0} + F'$.
\end{enumerate}
\end{coro}

\begin{proof}
  By Lemma~\ref{lem:wap-td-is-tb}, the measure $\mu$ is translation
  bounded.  Now, due to Proposition~\ref{prop:distri}, $\mu$ has
  unique autocorrelation $\gamma$ and diffraction
\[
  \widehat{\gamma} \, = \sum_{\chi \in S}
  \lvert c_\chi(\mu) \rvert^2 \, \delta^{}_{\chi} \ts ,
\]
where $S = \supp(\widehat{\gamma}\ts ) = \supp (\nu^{}_{\mathsf{pp}})$.

Then, by Theorem~\ref{thm:finally}, either claim $(1)$ holds, or $S$
is uniformly discrete. In the latter case, $S$ is $\cB$-sparse for all
van Hove sequences, and hence, by Corollary~\ref{coro:wap}, claim
$(2)$ holds.
\end{proof}

The explicit structure can then be summarised as follows.

\begin{coro}
  Let\/ $\mu$ and\/ $\nu$ be as in
  Corollary~\textnormal{\ref{coro:almost-done}}.  Then, there exists a
  CPS\/ $(\RR^d, H, \cL)$ and some\/ $h\in \Cc (H)$ such that the
  autocorrelation and diffraction of\/ $\mu$ are
\[
     \gamma \, = \, \dens (\cL) \, \omega^{}_{h * \widetilde{h}}
     \quad \text{and} \quad
     \widehat{\gamma} \, = \, \dens (\cL)^2 \ts
     \omega^{}_{\lvert \widecheck{h}\rvert^2} \ts ,
\]   
    and that the two cases are then as follows:
\begin{enumerate}\itemsep=2pt
\item 
   $\mu_{\mathsf{s}} = \omega^{}_{h}$ and\/
   $\nu^{}_{\mathsf{pp}} = \dens (\cL) 
    \, \omega^{}_{\widecheck{h}}$;
\item 
    $\mu = \omega^{}_{h}$ and\/ $\nu = \dens (\cL) 
     \, \omega^{}_{\widecheck{h}}$.
\end{enumerate}
Further, in the second case, $\mu$ and\/ $\nu = \widehat{\mu}$ have
the form given in Theorem~\textnormal{\ref{thm:trig-coeffs}}, with an
internal space of the form\/ $H=\ZZ^m \!\times\nts\KK$. \qed
\end{coro}

\section*{Acknowledgements}

{It is our pleasure to thank an anonymous referee for 
providing numerous careful comments that significantly
helped to improve the presentation.}
This work was supported by the German Research Foundation (DFG),
within the SFB 1283 at Bielefeld University, by the Natural Sciences
and Engineering Council of Canada (NSERC), via grant 03762-2014, and
by the Australian Research Council (ARC), via Discovery Project DP
180{\ts}102{\ts}215.  

\end{document}